\documentclass[letterpaper,11pt]{amsart}
\usepackage{amsmath,amsthm,amssymb,mathtools,amsfonts,mathrsfs}
\usepackage{xspace,xcolor}
\usepackage[breaklinks,colorlinks,citecolor=teal,linkcolor=teal,urlcolor=teal,pagebackref,hyperindex]{hyperref}
\usepackage[alphabetic]{amsrefs}
\usepackage{tikz-cd}
\usepackage[all]{xy}
\usepackage{bbm}
\usepackage{verbatim}

\newtheorem{thm}{Theorem}[section]
\newtheorem{prop}[thm]{Proposition}

\newtheorem{lem}[thm]{Lemma}
\newtheorem{cor}[thm]{Corollary}

\theoremstyle{definition}
\newtheorem{defn}[thm]{Definition}
\newtheorem{ex}[thm]{Example}
\newtheorem{rmk}[thm]{Remark}

\newtheorem{conv}{Convention}
\newtheorem*{nota*}{Notation}

\newcommand{\C}{\mathbb C}

\newcommand{\Z}{\mathbb Z}

\newcommand{\Pp}{\mathbb P}

\newcommand{\sO}{\mathcal O}

\newcommand{\bM}{\overline{\mathcal M}}
\newcommand{\bF}{\overline{\mathcal F}}

\newcommand{\partn}{\mathcal \pi}
\newcommand{\ualpha}{{\underline\alpha}}
\newcommand{\uepsilon}{{\underline\varepsilon}}

\newcommand{\mS}{\mathfrak S}

\newcommand{\cD}{\mathcal D}
\newcommand{\HH}{\mathfrak H}
\newcommand{\fG}{{\mathfrak G}}
\newcommand{\DR}{\text{DR}}

\newcommand{\on}{\operatorname} 
\newcommand{\g}{\mathrm{g}}
\newcommand{\ev}{\operatorname{ev}}
\newcommand{\HE}{\operatorname{HE}}
\newcommand{\rel}{\mathrm{rel}}
\newcommand{\vir}{\mathrm{vir}}
\newcommand{\eq}{\mathrm{eq}}
\newcommand{\GM}{\C^*}

\newcommand{\includegraphicsdpi}[3]{
    \pdfimageresolution=#1  
    \includegraphics[#2]{#3}
    \pdfimageresolution=72  
}

\title[Higher genus relative Gromov--Witten theory and DR-cycles]{Higher genus relative Gromov--Witten theory and double ramification cycles}

\author{Honglu Fan}
\email{honglu.fan@math.ethz.ch}
\author{Longting Wu}
\email{longting.wu@math.ethz.ch}
\author{Fenglong You}
\email{fenglong@ualberta.ca}

\begin{document}

\maketitle

\begin{abstract}
We extend the definition of relative Gromov--Witten invariants with negative contact orders to all genera. Then we show that relative Gromov--Witten theory forms a partial CohFT. Some cycle relations on the moduli space of stable maps are also proved.
\end{abstract}
\tableofcontents

\section{Introduction}

\subsection{Overview}

Gromov--Witten theory is the first modern approach in enumerative geometry. It can be viewed as a theory that virtually counts curves in smooth projective varieties. A lot of structural properties have been developed for Gromov--Witten theory since 1990s. On the other hand, relative Gromov--Witten theory naturally appears when one considers counting curves in a smooth projective variety $X$ with tangency conditions along a divisor $D$. Relative Gromov--Witten invariants are defined by Li--Ruan \cite{LR}, J. Li \cite{Jun1}, \cite{Jun2} and Ionel--Parker \cite{IP2} nearly two decades ago. However, structural properties of relative Gromov--Witten theory still have a lot to develop. 

The structures of genus-zero relative Gromov--Witten theory have been recently developed in \cite{FWY}. The main motivation for \cite{FWY} is the equality between genus-zero relative Gromov--Witten invariants of the smooth pair $(X,D)$ and the orbifold Gromov--Witten invariants of the $r$th root stack $X_{D,r}$ when $r$ is sufficiently large (see \cite{ACW} and \cite{TY}). Such equality indicates a possibility of converting known results of the orbifold theory into the relative theory. However, the results in \cite{ACW} and \cite{TY} only involve orbifold Gromov--Witten invariants with small ages. In \cite{FWY}, we introduced genus-zero relative invariants with negative contact orders and proved that they are equal to orbifold invariants with large ages (up to a suitable power of $r$). Several structural properties of genus-zero relative invariants can be proved using corresponding properties of orbifold invariants. These properties include relative quantum cohomology, topological recursion relation (TRR), WDVV equation, Givental's formalism and genus-zero Virasoro constraints.

This paper is a sequel to \cite{FWY}. We study the structures of higher genus relative Gromov--Witten invariants by introducing higher genus relative invariants with negative contact orders. While higher genus relative invariants and orbifold invariants are not equal in general, the result of \cite{TY} shows that orbifold invariants are polynomials in $r$ for a sufficiently large $r$ and relative invariants are equal to the constant terms of the orbifold invariants. The result in \cite{TY} only contains orbifold invariants with small ages. In this paper, we extend the result of \cite{TY} to orbifold invariants with large ages. The key to our results is the Theorem \ref{thm:appdx} relating the pushforward of DR cycles and Hurwitz--Hodge classes. It has two consequences. Firstly, it can be used to generalize the definition of genus zero relative theory in \cite{FWY} to all genera relative theory which satisfies a partial cohomological field theory (partial CohFT for short, see Section \ref{sec:pcohft} for more details). Secondly, it also generalizes \cite{CJ} to DR cycles with target varieties to produce some relations (see Section \ref{sec:cyc-rel} for more details).

\subsection{Relative Gromov--Witten theory}

We first focus on relative Gromov--Witten theory with possibly negative contact orders in the sense of \cite{FWY}. We quickly review the idea of \cite{FWY}. We presented two definitions of relative Gromov--Witten cycles. In the first definition, we simply defined relative virtual cycles as a limit of orbifold virtual cycles. In the second definition, we first defined a type of bipartite graphs. Given a bipartite graph $\fG$, we associated each vertex with a moduli space, and glued them up using fiber products according to edges, thus forming a stack $\bM_\fG$. It also came with an ``obstruction class" $\iota^*C_\fG\in A^*(\bM_\fG)$ (for $\iota$, see equation \eqref{eqn:diag}). And finally, we defined the relative Gromov--Witten cycle with negative contact as the pushforward of $\iota^*C_\fG\cap[\bM_\fG]^{\vir}$ divided by the number of automorphisms to the corresponding moduli of stable maps.

In \cite{FWY}, we already laid down the set-up of bipartite graphs for all genera. It is natural to expect higher genus graphs lead to higher genus relative Gromov--Witten cycles. However, at that time we still made genus-zero restrictions on the bipartite graphs in the main results because no evidence was found to justify the usefulness of the higher genus generalization. In this paper, we aim to present progress in the higher genus generalization. The main point of this paper can be summarized as follows:
\begin{quote}
    If we simply apply the same definition of relative Gromov--Witten cycles in \cite{FWY} without genus-zero constraints on vertices and admissible bipartite graphs, the resulting all-genus relative theory with negative contact orders would satisfy a partial CohFT in the sense of \cite{LRZ}.
\end{quote}
Ultimately, we look for all structures in the relative Gromov--Witten theory, and hopefully they might help us deal with long-standing questions in this field (for example, how to write down a Virasoro operator for relative theory). As a very subjective remark, the results in this paper perhaps suggest that extending \cite{FWY} to higher genera might be a reasonable idea in the search of broader structures. According to \cite{LRZ}, a partial CohFT is a CohFT without the loop gluing axiom. Finding a replacement of this loop gluing axiom could also be an interesting question, which further relates with how to extend higher genus Givental quantization and how to write down all genus Virasoro operators.

\subsection{Double ramification cycles}

Double ramification cycles on the moduli space of curves are tautological classes and can be defined using relative Gromov--Witten theory of $(\mathbb P^1,0\cup\infty)$. An explicit formula for double ramification cycles is given in \cite{JPPZ}. An introduction to 
tautological classes and double ramification cycles can be found in \cite{P}.

Recently, double ramification cycles with target varieties have been studied in \cite{JPPZ18}. Given a line bundle $L$ over $D$, the double ramification cycle $\DR_\Gamma(D,L)$ with the target variety $D$ is defined using the moduli space $\bM_{\Gamma}^\sim(D)$ of relative stable maps to rubber targets over $D$.
By \cite{JPPZ18}, the double ramification cycle $\DR_\Gamma(D,L)$ is equal to the constant term of the polynomial class $P_{\Gamma}^{g,r}(D,L)$, see Section \ref{sec:cyc-rel} for the definition. 

More generally, the polynomial class $P_{\Gamma}^{d,r}(D,L)$ is a tautological class in the moduli space of stable maps to $D$. When the target variety $D$ is a point, it is conjectured by A. Pixton that the constant term of $P_{\Gamma}^{d,r}(D,L)$ vanishes for $d>g$. When the target variety $D$ is a point, the conjecture is proved in \cite{CJ}. The proof of \cite{CJ} has been generalized to general targets in \cite{Bae19}. 

The formula for double ramification cycles in \cite{JPPZ} and \cite{JPPZ18} are obtained by relating it to certain Hurwitz--Hodge cycles which are also polynomials in $r$ and the constant terms coincide with the constant terms of $P_{\Gamma}^{d,r}(D,L)$. In Section \ref{sec:HH-DR}, we generalize the identity between Hurwitz--Hodge cycles and double ramification cycles. As a direct consequence, Hurwitz--Hodge cycles satisfy a refined polynomiality (see Corollary \ref{cor:cycrel}) which implies the vanishing theorem for $P_{\Gamma}^{d,r}(D,L)$ in \cite{CJ} and \cite{Bae19}. 

\subsection{Acknowledgment}
We would like to thank Rahul Pandharipande and Qile Chen for helpful discussions and valuable comments. H. F. is supported by grant ERC-2012-AdG-320368-MCSK and SwissMAP. L. W. is supported by grant ERC-2017-AdG-786580-MACI. F. Y. is supported by the postdoctoral fellowship of NSERC and Department of Mathematical Sciences at the University of Alberta. 
F.Y. is also supported by a postdoctoral fellowship of the Fields Institute for Research in Mathematical Sciences.

This project has received funding from the European Research Council (ERC) under the European Union’s
Horizon 2020 research and innovation program (grant agreement No. 786580).

\section{Relative and orbifold Gromov--Witten theories}
In this section, we first give a brief review of relative Gromov--Witten theory defined by Li--Ruan \cite{LR}, J. Li \cite{Jun1}, \cite{Jun2} and Ionel--Parker \cite{IP2}. We then also review orbifold Gromov--Witten theory following \cite{AGV02}, \cite{AGV}, \cite{CR} and \cite{Tseng}. For the orbifold theory, we only focus on root stacks and gerbes as target spaces.

\subsection{Relative theory}

Let $X$ be a smooth projective variety and $D$ a smooth divisor. The intersection number of a curve class $\beta$ with a divisor $D$ is denoted by $\int_\beta D$. 

A \emph{topological type} $\Gamma$ is a tuple $(g,n,\beta,\rho,\vec{\mu})$ where $g,n$ are non-negative integers, $\beta\in H_2(X,\Z)$ is a curve class and $\vec{\mu}=(\mu_1,\dotsc,\mu_\rho)\in \Z^\rho$ is a partition of the number $\int_\beta D$. Furthermore, we must have
\begin{equation}\label{eqn:positivecontact}
    \mu_i>0 \text{ for } 1\leq i\leq \rho.
\end{equation}
Let $\bM_\Gamma(X,D)$ be the moduli of relative stable maps. There is a stabilization map $\mathfrak s:\bM_\Gamma(X,D)\rightarrow \bM_{g,n+\rho}(X,\beta)$. Write $\bar\psi_i=\mathfrak s^*\psi_i$. There are evaluation maps
\begin{align*}
\ev_X=(\ev_{X,1},\ldots,\ev_{X,n}):&\bM_\Gamma(X,D)\rightarrow X^n, \\
\ev_D=(\ev_{D,1},\ldots,\ev_{D,\rho}):&\bM_\Gamma(X,D)\rightarrow D^\rho.
\end{align*}

The insertions of relative invariants are the following classes.
\[
\ualpha\in (\C[\bar\psi]\otimes H^*(X))^{n}, \quad \uepsilon\in (\C[\bar\psi]\otimes H^*(D))^{\rho}.
\]
For simplicity, we assume
\[
\ualpha=(\bar\psi^{a_1}\alpha_1,\ldots,\bar\psi^{a_n}\alpha_n), \quad \uepsilon=(\bar\psi^{b_1}\varepsilon_1,\ldots,\bar\psi^{b_\rho}\varepsilon_\rho).
\]
The \emph{relative Gromov--Witten invariant with topological type $\Gamma$} is defined to be
\[
\langle \uepsilon \mid \ualpha \rangle_{\Gamma}^{(X,D)}=\displaystyle\int_{[\bM_{\Gamma}(X,D)]^{\on{vir}}} \ev_D^*\uepsilon \cup \ev_X^*\ualpha,
\]
where
\begin{equation}\label{eqn:ev}
\ev_D^*\uepsilon=\prod\limits_{j=1}^\rho \bar\psi^{b_j}_{D,j}\ev_{D,j}^*\varepsilon_j, \quad \ev_X^*\ualpha=\prod\limits_{i=1}^n \bar\psi^{a_i}_{X,i}\ev_{X,i}^*\alpha_i,
\end{equation}
with $\bar\psi_{D,j}, \bar\psi_{X,i}$ the psi-classes of the corresponding markings.

We also allow disconnected domains. Let $\Gamma=\{\Gamma^\partn\}$ be a set of topological types. The relative invariant with disconnected domain curves is defined by the product rule:
\[
\langle \uepsilon\mid \ualpha \rangle_{\Gamma}^{\bullet(X,D)} = \prod\limits_{\partn} \langle \uepsilon^{\partn} \mid \ualpha^{\partn} \rangle_{\Gamma^{\partn}}^{(X,D)}.
\]
Here $\bullet$ means possibly disconnected domains. We will call this $\Gamma$ a \emph{disconnected topological type}. We now recall the definition of an {\emph {admissible graph}}.

\begin{defn}[Definition 4.6, \cite{Jun1}]\label{defn:adm}
{\emph {An admissible graph}} $\Gamma$ is a graph without edges plus the following data.
\begin{enumerate}
    \item An ordered collection of legs.
    \item An ordered collection of weighted roots.
    \item A function $\g:V(\Gamma)\rightarrow \Z_{\geq 0}$.
    \item A function $b:V(\Gamma)\rightarrow H_2(X,\Z)$.
\end{enumerate}
\end{defn}
Here, $V(\Gamma)$ means the set of vertices of $\Gamma$. Legs and roots are regarded as half-edges of the graph $\Gamma$. 
A relative stable morphism is associated to an admissible graph in the following way. Vertices in $V(\Gamma)$ correspond to the connected components of the domain curve. Roots and legs correspond to relative markings and interior markings, respectively. Weights on roots correspond to contact orders at the corresponding relative markings. 

The functions $\g,b$ assign a component to its genus and degree, respectively. We do not spell out the formal definitions in order to avoid heavy notation, but we refer the readers to \cite{Jun1}*{Definition 4.7}. 

\begin{rmk}
A (disconnected) topological type and an admissible graph are equivalent concepts. Different terminologies emphasize different aspects. For example, admissible graphs will be glued at half-edges into actual graphs.
\end{rmk}


\subsection{Rubber theory}
Relative Gromov--Witten theory is closely related with the so-called rubber theory. Given a smooth projective variety $D$ and a line bundle $L$ on $D$, we denote the \emph{moduli of relative stable maps to rubber targets} by $\bM^{\bullet\sim}_{\Gamma'}(D)$. Here $\bullet$ means possibly disconnected domains, and $\sim$ means rubber targets. The discrete data $\Gamma'$ describing the topology of relative stable maps is defined as a slight variation of an admissible graph.

In the rest of the paper, it is very often that a pair $(X,D)$ is given in the context. In this case, we always assume that
\[
L=N_{D/X}.
\]

\begin{defn}\label{def:rubber}
A rubber (admissible) graph $\Gamma'$ is an admissible graph whose roots have two different types. There are
\begin{enumerate}
    \item $0$-roots (whose weights will be denoted by $\mu^0_1,\ldots,\mu^0_{\rho_0}$), and
    \item $\infty$-roots (whose weights will be denoted by $\mu^\infty_1,\ldots,\mu^\infty_{\rho_\infty}$).
\end{enumerate}
Furthermore, the curve class assignment $b$ maps $V(\Gamma)$ to $H_2(D,\Z)$.
\end{defn}

As to the moduli space of relative stable maps to a rubber (non-rigid) target, a description can be found in, for example, \cite{GV}*{Section 2.4}. Briefly speaking, a relative stable map to a rubber target of $D$ is a relative pre-stable map to a chain of $\Pp_D(L\oplus \sO)$ glued along certain invariant sections. We denote the invariant divisors at two ends of the chain by $D_0, D_\infty$. We make the convention that the normal bundles of $D_0$ and $D_\infty$ are $L$ and $L^\vee$, respectively. 

We also have evaluation maps
\[
\ev_D:\bM^{\bullet\sim}_{\Gamma'}(D)\rightarrow D^n,\quad  \ev_{D_0}:\bM^{\bullet\sim}_{\Gamma'}(D)\rightarrow D^{\rho_0},\quad \ev_{D_\infty}:\bM^{\bullet\sim}_{\Gamma'}(D)\rightarrow D^{\rho_\infty}.
\]

To get a non-empty moduli space, we need the following condition.
\begin{equation}\label{eqn:rubberint}
    \sum\limits_{i=1}^{\rho_0} \mu^0_i - \sum\limits_{j=1}^{\rho_\infty} \mu^\infty_j = \int_\beta c_1(L),
\end{equation}
where $\beta$ is the curve class of $\Gamma'$. If $\Gamma'$ has more than one vertex, the above is satisfied on each of the vertex (with $\mu^0_i,\mu^\infty_j$ corresponding to weights of roots on a given vertex).


A relative stable map to a rubber target is associated to a rubber graph in the standard way, with relative markings at $D_0$ and $D_\infty$ corresponding to the $0$-roots and $\infty$-roots, respectively.


\subsection{Orbifold theory of root stacks and gerbes}
We consider the $r$th root stack $X_{D,r}$ of $X$ along the divisor $D$. We write the coarse moduli space of the inertia stack of $X_{D,r}$ as $\underline{I}(X_{D,r})$. It has $r$ components:
\[
\underline{I}(X_{D,r})  \cong X\sqcup D\sqcup D\cdots\sqcup D.
\]
The twisted sectors which are isomorphic to $D$ are labeled by the ages $k_i/r$, where $k_i\in\{1,2,\ldots,r-1\}$.

Since our task is to compare orbifold theory with relative theory, we would like to match some of their notation. Let $\Gamma=(g,n,\beta,\rho,\vec{\mu})$ be a topological type with $\vec{\mu}=(\mu_1,\dotsc,\mu_\rho)\in (\Z^*)^\rho$ being a partition of the number $\int_\beta D$. A topological type $\Gamma$ can also be used to specify topological types of orbifold stable maps to $X_{D,r}$ under the following convention. 
\begin{conv}\label{conv:Gamma}
A topological type $\Gamma=(g,n,\beta,\rho,\vec{\mu})$ of orbifold stable maps contains the following data:
\begin{itemize}
    \item $g,\beta$ corresponds to the genus and curve class;
    \item $n$ indicates a set of $n$ markings without orbifold structure;
    \item $\rho$ indicates a set of $\rho$ markings with orbifold structure;
    \item $\vec{\mu}=(\mu_1,\dotsc,\mu_\rho)\in (\Z^*)^\rho$ and $\sum_{i=1}^{\rho}\mu_i=\int_d D$;
    \item When $\mu_i>0$, the evaluation map of the corresponding marking lands on the twisted sector with age $\mu_i/r$;
    \item when $\mu_i<0$, the evaluation map of the corresponding marking lands on the twisted sector with age $(r+\mu_i)/r$.
\end{itemize} 
Here, we require that $r>\max_{1\leq i\leq \rho}|\mu_i|$.
\end{conv}

A priori, evaluation maps should land on $\underline{I}(X_{D,r})$. Since ages are fixed by $\vec{\mu}$, we further restrict their targets to the corresponding components. We have the restricted evaluation maps
\[
\ev_X:\bM_\Gamma(X_{D,r})\rightarrow X^n, \quad \ev_D:\bM_\Gamma(X_{D,r})\rightarrow D^\rho
\]
corresponding to those $n$ markings without orbifold structures, and those $\rho$ markings with orbifold structures, respectively. We similarly denote their entries by
\[
\ev_X = (\ev_{X,1},\ldots,\ev_{X,n}), \quad \ev_D = (\ev_{D,1},\ldots,\ev_{D,\rho}).
\]

Consider the forgetful map
\[
\tau: \bM_\Gamma(X_{D,r})\rightarrow \bM_{g,n+\rho}(X,\beta)\times_{X^\rho} D^{\rho},
\]
we write $\bar{\psi_i}=\tau^*\psi_i$.
Using the notation in the Definition \ref{rel-inv-neg}, the \emph{orbifold Gromov--Witten invariant with topological type $\Gamma$} is

\begin{equation}\label{orb-inv}
    \langle \uepsilon, \ualpha \rangle_{\Gamma}^{X_{D,r}}=\displaystyle\int_{[\bM_{\Gamma}(X_{D,r})]^{\vir}} \prod\limits_{j=1}^\rho \bar\psi_{D,j}^{b_j}\ev_{D,j}^*\varepsilon_j\prod\limits_{i=1}^n \bar\psi_{X,i}^{a_i}\ev_{X,i}^*\alpha_i,
\end{equation}
where $\bar\psi_{D,j}, \bar\psi_{X,i}$ are psi-classes corresponding to markings evaluated under $\ev_D, \ev_X$.

For orbifold Gromov--Witten invariants (\ref{orb-inv}) with topological type $\Gamma$, we define an integer $\rho_-\in \mathbb Z_{\geq 0}$ to be
\begin{align}\label{rho-neg}
\rho_-=\sum_{\mu_i>0}\mu_i/r+\sum_{\mu_i<0}(r+\mu_i)/r - \left( \int_\beta D \right)/r.
\end{align}
When $r$ is sufficiently large, virtual dimension being an integer implies that $\rho_-$ equals the number of negative parts in $\vec{\mu}$. 

\section{Relative theory with negative contact orders in all genera}

\subsection{Relative theory as a limit of orbifold theory}
Following \cite{FWY}, we define relative Gromov--Witten theory as a limit of orbifold Gromov--Witten theory. However, contrary to the genus-zero invariants, higher genus orbifold invariants of the root stack $X_{D,r}$ depend on $r$. According to \cite{TY}, orbifold invariants of root stacks are polynomials in $r$ and the constant terms are the corresponding relative invariants when $r$ is sufficiently large. Combining the ideas in \cite{TY} and \cite{FWY}, we give our first definition of relative Gromov--Witten invariants (possibly with negative contact orders) as the lowest term of orbifold invariants for a sufficiently large $r$.

\begin{thm}\label{thm:limitexist}
Fix a topological type $\Gamma=(g,n,\beta,\rho,\vec{\mu})$. For a sufficiently large $r$, the following cycle class
\[
r^{\rho_-}\tau_*\left([\bM_\Gamma(X_{D,r})]^{\on{vir}}\right)\in A_*(\bM_{g,n+\rho}(X,\beta) \times_{X^\rho} D^{\rho})
\]
is a polynomial in $r$, where $\tau$ is the forgetful map
\[
\tau: \bM_\Gamma(X_{D,r})\rightarrow \bM_{g,n+\rho}(X,\beta) \times_{X^\rho} D^{\rho}.
\]
\end{thm}
Theorem \ref{thm:limitexist} will be proven in Section \ref{sec:match}. For a sufficiently large $r$, denote the constant term of the above cycle by \[\mathop{\text{lim}}_{r\rightarrow \infty} \left[r^{\rho_-}\tau_*([\bM_\Gamma(X_{D,r})]^{\vir})\right]_{r^0}.\]

\begin{defn}\label{def:cycle}
Fix a topological type $\Gamma=(g,n,\beta,\rho,\vec{\mu})$. The \emph{relative Gromov--Witten cycle} of the topological type $\Gamma$ is defined as
\[
\mathfrak c_\Gamma(X/D) = \mathop{\text{lim}}_{r\rightarrow \infty}\left[r^{\rho_-}\tau_*([\bM_\Gamma(X_{D,r})]^{\vir})\right]_{r^0} \in A_*(\bM_{g,n+\rho}(X,\beta) \times_{X^\rho} D^{\rho}).
\]
\end{defn}

When $\rho_-=0$, $\mathfrak c_\Gamma(X/D)$ coincides with the pushforward of virtual cycle from the corresponding moduli of relative stable maps.

Recall that $\rho$ is used to denote the number of roots in the admissible graph $\Gamma$. Let $\rho_+$ be the number of roots whose weights are positive. We have
\begin{prop}\label{prop:vdim}
\[\mathfrak c_\Gamma(X/D) \in A_{d}(\bM_{g,n+\rho}(X,\beta) \times_{X^\rho} D^{\rho}),\] where 
\[d=(1-g)(\mathrm{dim}_{\mathbb C}(X)-3)+\int_{\beta} c_1(T_X(-\mathrm{log} D)) + n + \rho_+. \]
\end{prop}

We define relative Gromov--Witten invariants (possibly with negative contact orders) by integrations against this cycle. 
The insertions are
\begin{align*}
    \begin{split}
        \ualpha = (\bar\psi^{a_1}\alpha_1,\ldots,\bar\psi^{a_n}\alpha_n) 
        &\in (\C[\bar\psi] \otimes H^*(X))^{n},\\ \uepsilon = (\bar\psi^{b_1}\epsilon_1,\ldots,\bar\psi^{b_{\rho}}\epsilon_\rho) &\in (\C[\bar\psi] \otimes H^*(D))^{\rho}.
    \end{split}
\end{align*}
There are two types of evaluation maps from $\bM_\Gamma(X_{D,r})$ corresponding to interior markings and relative markings respectively:
\begin{align*}
\ev_X=(\ev_{X,1},\ldots,\ev_{X,n}):\bM_\Gamma(X_{D,r})&\rightarrow X^n, \\
\ev_D=(\ev_{D,1},\ldots,\ev_{D,\rho}):\bM_\Gamma(X_{D,r})&\rightarrow D^\rho.
\end{align*}
Recall that we have the stabilization map $\tau:\bM_\Gamma(X_{D,r})\rightarrow \bM_{g,n+\rho}(X,\beta) \times_{X^\rho} D^{\rho}$. There are also evaluation maps
\begin{align*}
\overline{\ev}_X=(\overline\ev_{X,1},\ldots,\overline\ev_{X,n}):\bM_{g,n+\rho}(X,\beta) \times_{X^\rho} D^{\rho}&\rightarrow X^n, \\
\overline{\ev}_D=(\overline\ev_{D,1},\ldots,\overline\ev_{D,\rho}):\bM_{g,n+\rho}(X,\beta) \times_{X^\rho} D^{\rho}&\rightarrow X^\rho,
\end{align*}
such that 
\[
\overline{\ev}_X\circ \tau=\ev_X, \quad \overline{\ev}_D\circ\tau=\ev_D.
\]
\begin{defn}\label{rel-inv-neg1}
The {\it relative Gromov--Witten invariant of topological type $\Gamma$ with insertions $\uepsilon,\ualpha$} is
\[
\langle \uepsilon \mid \ualpha \rangle_{\Gamma}^{(X,D)} =  \displaystyle\int_{\mathfrak c_\Gamma(X/D)} \prod\limits_{j=1}^{\rho} \bar{\psi}_{D,j}^{b_j}(\overline{\ev}_{D,j})^*\epsilon_j\prod\limits_{i=1}^n \bar{\psi}_{X,i}^{a_i}\overline{\ev}_{X,i}^*\alpha_i,
\]
where $\bar\psi_{D,j}, \bar\psi_{X,i}$ are pull-backs of psi-classes on $\bM_{g,n+\rho}(X,\beta)$ corresponding to relative and interior markings.
\end{defn}

As a direct consequence, we have the following relation between relative and orbifold invariants.

\begin{thm}\label{thm:limit}
Fix a topological type $\Gamma=(g,n,\beta,\rho,\vec{\mu})$. For $r\gg 1$, the following equality holds:
\begin{equation}
    \left[r^{\rho_-}\langle \uepsilon, \ualpha \rangle_{\Gamma}^{Y_{D,r}}\right]_{r^0}=\langle\uepsilon  \mid \ualpha \rangle_{\Gamma}^{(Y,D)},
\end{equation}
where the relative invariant on the right hand side of the above equation is the relative invariant with negative contact order defined in Definition \ref{rel-inv-neg1} when $\rho_->0$.
\end{thm}

\subsection{Graph notations}
Following \cite{FWY}, the second definition of the relative Gromov--Witten cycle is given by a graph sum. In this section, we review the graph notation. In \cite{FWY}*{Section 4}, we defined certain bipartite graphs and built moduli spaces based on those graphs. In higher genera, we in fact use the same notion of bipartite graphs except that graphs may now have loops. Readers familiar with \cite{FWY}*{Section 4} should keep in mind that we no longer have the genus constraint on admissible bipartite graphs and could skip to later sections. But to make a relatively self-contained treatment, here we try to make a concise (and hopefully more understandable) summary of graph notation in \cite{FWY}*{Section 4}. 

\subsubsection{Two sides of bipartite graph}
The main type of bipartite graphs is called \emph{admissible bipartite graph}. A bipartite graph divides vertices into two sides, with edges connecting back and forth from one side to the other (but not between vertices on the same side). In this paper, we label the two sides differently as \emph{$0$-side} and \emph{$\infty$-side}.

\subsubsection{Vertices on $0$-side}
Vertices on the two sides have different decorations. We treat a single vertex on $0$-side along with its decorations as a graph by itself (denoted by $\Gamma^0_i$ with $i$ in a label set). It is called a \emph{graph of type $0$}. Recall the rubber admissible graph is defined in Definition \ref{def:rubber}. A graph of type $0$ is in fact a rubber admissible graph, with some refinements on the types of roots. Formally, we have the following definition.
\begin{defn}\label{def:admgraph0}
\emph{A (connected) graph of type $0$} is a weighted graph $\Gamma^0$ consisting of a
single vertex, no edges, and the following four types of half-edges:
\begin{enumerate}
\item {$0$-roots},
\item {$\infty$-roots of node type},
\item {$\infty$-roots of marking type},
\item {Legs}.
\end{enumerate}
$0$-roots are weighted by
positive integers, and $\infty$-roots are weighted by negative integers. The vertex is associated to a tuple $(g,\beta)$ where $g\geq 0$
and $\beta\in H_2(D,\Z)$. 
\end{defn}

\subsubsection{Vertices on $\infty$-side}
On the other hand, we combine all vertices on $\infty$-side into one single (possibly disconnected) graph $\Gamma^\infty$. It forms an admissible graph (see Definition \ref{defn:adm}) where roots are further distinguished by 
\begin{enumerate}
    \item node type,
    \item marking type.
\end{enumerate}

\subsubsection{Admissible bipartite graph}
An admissible bipartite graph $\mS$ is defined as a tuple $(\mS_0,\Gamma^\infty,I,E,\g,b)$. We provide a more descriptive explanation than \cite{FWY}*{Definition 4.8} in the following.

\begin{enumerate}
    \item $\mS_0=\{\Gamma_i^0\}$ is a set of graphs of type $0$. $\Gamma^\infty$ is a (possibly disconnected) graph of type $\infty$.
    
    \item \emph{(Edges)} $E$ is the set of edges each of which identifies an $\infty$-root of node type in a vertex of $\mS_0$, and a root of node type in a vertex of $\{\Gamma^\infty\}$.
    
    \item \emph{(Marking labeling)} $I$ is a one-one correspondence between the set $\{1,\ldots, n+\rho\}$ and the set of all half-edges in $\mS_0 \sqcup \{\Gamma^\infty\}$ except $0$-roots of node type (in all $\Gamma^0_i$) and roots of note type(in $\Gamma^\infty$). For the convenience of further treatment, we assume:
    \begin{quote}
        the set of all legs corresponds to the subset $\{1,\ldots,n\}$.
    \end{quote}
    
    \item \emph{(Genus assignment)} $\g$ is a map from the set of vertices to $\Z_{\geq 0}$. The value of $\g$ at a vertex is exactly the genus decoration on the vertex in the original rubber graph or admissible graph.
    
    \item \emph{(Degree assignment)} $b$ is a map sending the set of vertices in $\mS_0$ to $H_2(D,\Z)$, and the set of vertices in $\Gamma^\infty$ to $H_2(X,\Z)$. The value of $b$ at a vertex is exactly the curve class decoration on the vertex in the original rubber graph or admissible graph.
\end{enumerate}
In addition, $\mS$ satisfies the following.
\begin{enumerate}
     \item For a graph $\Gamma^0_i$, the sum of weights of all roots equals $\int_{b(\Gamma^0_i)} D$. For a vertex in $\Gamma^\infty$, the sum of weights of all roots also equals the intersection of curve class with $D$.
    \item For every edge connecting an $\infty$-root of node type $l$ in a vertex $\Gamma^0_i$ and a root of node type $l'$ in a vertex of $\{\Gamma^\infty\}$, the weights of $l$ and $l'$ add up to $0$.
    \item All the vertices are \emph{stable}. We call a vertex $v$ \emph{stable} if either $b(v)\neq 0$ or the number of half-edges associated with $v$ is bigger than $2-2g(v)$.
\end{enumerate}

Given an admissible bipartite graph $\fG$, we sometimes talk about its topological type as a whole. The topological type of $\fG$ is a tuple $(g,n,\beta,\rho,\vec{\mu})$ where 
\begin{itemize}
    \item $g$ is the sum of genera (values of $\g$) of all vertices plus $h^1(\fG)$ (as a $1$-dimensional CW-complex),
    \item $\beta\in H_2(X,\Z)$ is the sum of curve classes (values of $b$) of all vertices (curve classes of $\Gamma^0_i$ in $H_2(D,\Z)$ are pushed forward to $H_2(X,\Z)$),
    \item $n$ is the number of legs,
    \item $\rho$ is the total number of $0$-roots, $\infty$-roots of marking type in $\mS_0$ and roots of marking type in $\Gamma^{\infty}$,
    \item and $\vec{\mu}$ is the list of weights of $0$-roots, $\infty$-roots of marking type in $\mS_0$ and roots of marking type in $\Gamma^{\infty}$.
\end{itemize}  

\begin{ex}
An example of an admissible bipartite graph is the following:

\includegraphicsdpi{600}{}{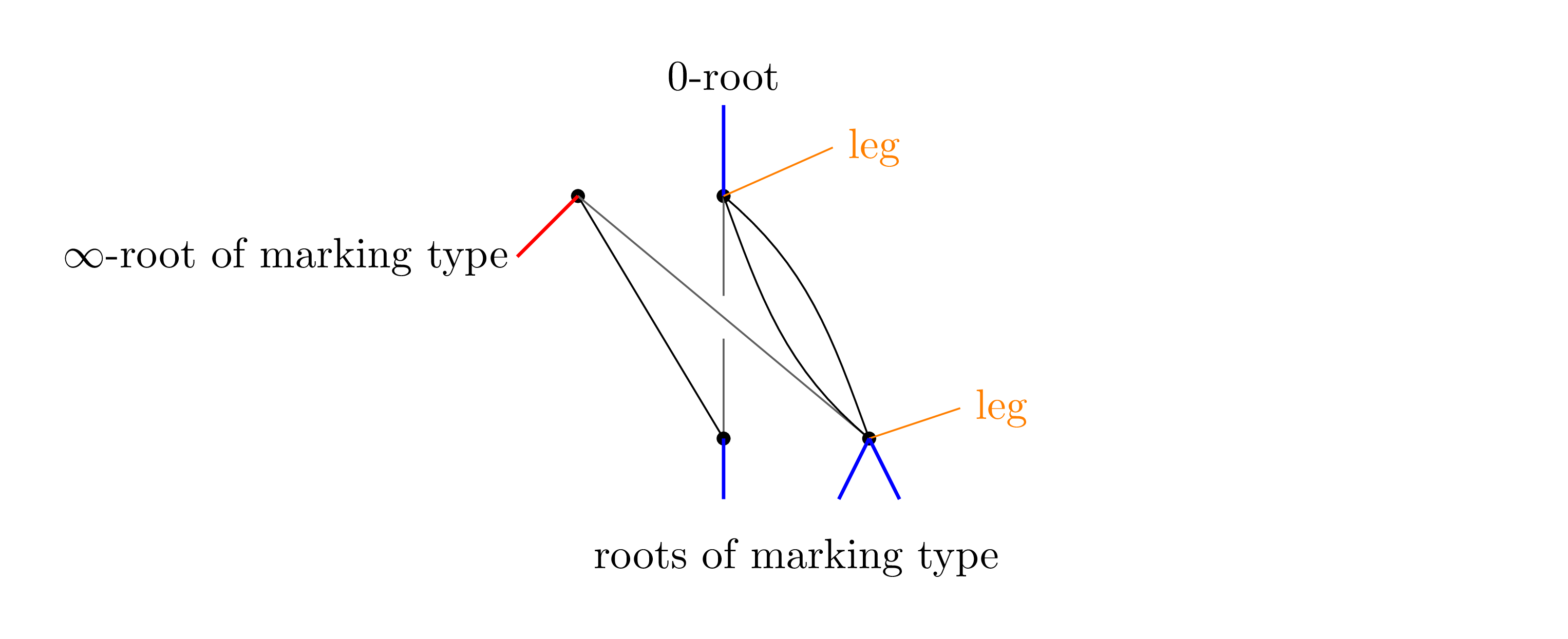}

Vertices on the top are on the $0$-side, and vertices on the bottom are on the $\infty$-side. Roots of marking type merge into edges. Moreover, $0$-side vertices can have $0$-roots, legs and $\infty$-roots of marking type, while $\infty$-side vertices can have legs and roots of marking type. We omit decorations of genus and curve class on the picture.
\end{ex}

\subsubsection{A graph of type $0$ corresponds to a moduli of rubber theory.}
Given a graph $\Gamma^0$ of type $0$, we associate a moduli space of relative stable maps with it. Recall that in the relative theory with rubber (non-rigid) targets over $D$, the target expands as a chain of $P_D(L\oplus \sO)$ glued along suitable sections. Two distinguished sections are denoted by $D_0$ and $D_\infty$. Our convention is that the normal bundle of $D_0$ is $L$.
\begin{defn}\label{def:typ0}
Define $\bM^{\sim}_{\Gamma^0}(D)$ to be the moduli stack of genus $g$, degree $\beta$ relative stable maps to rubber targets over $D$. Each marking corresponds to a half-edge of $\Gamma^0$ with the following assignments of contact orders: A $0$-root of weight $i$ corresponds to a relative marking over $D_0$ with contact order $i$. An $\infty$-root (of either type) of weight $i$ corresponds to a relative marking over $D_\infty$ of contact order $-i$ (recall $i$ is negative). A leg corresponds to an interior marking.
\end{defn}
\subsubsection{A graph of type $0$ corresponds to a moduli of orbifold theory.}
For technical purposes, we would also associate a moduli of orbifold stable maps to $\Gamma^0$ at some point.
\begin{defn}
Define $\bM_{\Gamma^0}(\cD_0)$ to be the moduli stack of genus $g$, degree $\beta$ orbifold stable maps to $\cD_0$ whose markings correspond to half-edges of $\Gamma^0$ with the following assignments of ages:
A $0$-root of weight
$i$ corresponds to a marking of age $i/r$. An $\infty$-root (of either type) of weight $i$ corresponds to a
marking of age $(r+i)/r$ (recall $i$ is negative in this case). A leg corresponds to a marking of age $0$ (that is, without orbifold structure).
\end{defn}

\subsubsection{A graph of type $\infty$ corresponds to a moduli of relative theory.}
Given a graph of type $\infty$, it naturally corresponds to a moduli of relative stable maps because the graph is an admissible graph to begin with.
\begin{defn}
Define $\bM^{\bullet}_{\Gamma^\infty}(X,D)$ to be the moduli of relative stable maps with topological type $\Gamma^\infty$. In $\Gamma^\infty$, roots of node type and roots of marking type are all treated as roots in this definition.
\end{defn}

\subsection{Relative theory as a graph sum}\label{sec:rel-inv}
We provide the second definition of the relative Gromov--Witten cycle of $(X,D)$ in this section. We would like to point out that the definition has in fact \emph{almost the same} appearance as the one in \cite{FWY}*{Section 5}. But the reader should keep in mind that the definition now allows higher genera in moduli, as well as loops in the bipartite graph. This is slightly more than a simple-minded extension, as the whole theory now satisfies a partial CohFT (see Section \ref{sec:pcohft}).

Let $X$ be a smooth projective variety and $D$ be a smooth divisor. Let the topological type $\Gamma=(g,n,\beta,\rho,\vec{\mu})$ with $\vec{\mu}=(\mu_1,\dotsc,\mu_\rho)\in (\Z^*)^\rho$ satisfying 
\[
\sum\limits_{i=1}^\rho \mu_i = \int_\beta D. 
\]
Similar to \cite{FWY}*{Section 5.1}, given a bipartite graph $\fG\in \mathcal B_\Gamma$, we consider the fiber product
\begin{equation}\label{eqn:fiberprod}
\bM_{\fG} = \prod\limits_{\Gamma^0_i\in \mS_0}\bM_{\Gamma^0_i}^\sim(D) \times_{D^{|E|}} \bM^{\bullet}_{\Gamma^\infty}(X,D),
\end{equation}
where the fiber product identifies evaluation maps according to edges $E$. That is, we have the following diagram (\cite{FWY}*{Diagram (5.2)}):
\begin{equation}\label{eqn:diag}
\xymatrix{
\bM_{\fG} \ar[r]^{} \ar[d]^{\iota} & D^{|E|} \ar[d]^{\Delta} \\
\prod\limits_{\Gamma^0_i\in \mS_0}\bM_{\Gamma^0_i}^\sim(D) \times \bM^{\bullet}_{\Gamma^\infty}(X,D) \ar[r]^{} & D^{|E|}\times D^{|E|}.
}
\end{equation}
The virtual fundamental class is
\[
[\bM_{\fG}]^{\on{vir}}=\Delta^!\left[\prod\limits_{\Gamma^0_i\in \mS_0}\bM_{\Gamma^0_i}^\sim(D)\times \bM^{\bullet}_{\Gamma^\infty}(X,D)\right]^{\vir}.
\]
There is a map
\[
\mathfrak t_{\fG}:\bM_{\fG}\rightarrow \bM_{g,n+\rho}(X,\beta) \times_{X^\rho} D^{\rho}
\]
obtained by the composition of the stabilization map and the boundary map which is described in \cite{FWY}*{Section 5.1}.

Following \cite{FWY}, to define relative Gromov--Witten cycle, we need to consider the following cycle classes

\begin{enumerate}
   \item For the graph $\Gamma^\infty$, we consider the class
    \begin{align}\label{neg-rel-infty}
C_{\Gamma^\infty}(t)=\dfrac{t}{t+\Psi}\in A^*(\bM^\bullet_{\Gamma^\infty}(X,D))[t^{-1}].
\end{align}
    \item For each graph $\Gamma^0_i$, we consider the class
    \begin{align}\label{neg-rel-0}
C_{\Gamma_i^0}(t)= \dfrac{\sum_{l\geq \g(i)}c(l-\g(i)) t^{\g(i)+\rho_\infty(i)-1-l}}{\prod\limits_{e\in \HE_{n}(\Gamma_i^0)} \big(\frac{t+\ev_e^*D}{d_e}-\bar\psi_e\big) } \in A^*(\bM_{\Gamma_i^0}^\sim(D))[t,t^{-1}].
\end{align}
The notation is explained as follows
\begin{itemize}
    \item the class
$c(l)=\Psi_\infty^l-\Psi_\infty^{l-1}\sigma_1+\ldots+(-1)^l\sigma_l$;
\item $\Psi_\infty$ is the divisor corresponding to the cotangent line bundle determined by the relative divisor on $\infty$ side;
\item $\sigma_k=\sum\limits_{\{e_1,\ldots,e_k\}\subset \HE_{m,n}(\Gamma_i^0)} \prod\limits_{j=1}^{k} (d_{e_j}\bar\psi_{e_j}-\ev_{e_j}^*D)$;
\item $\rho_\infty(i)$ is the number of $\infty$-roots (of both types) associated to $\Gamma_i^0$.
\end{itemize}
\end{enumerate}
    
For each $\fG$, we write
\begin{equation}\label{eqn:cg}
C_{\fG}=\left[ p_{\Gamma^\infty}^*C_{\Gamma^\infty}(t)\prod\limits_{\Gamma_i^0\in \mS_0} p_{\Gamma_i^0}^*C_{\Gamma_i^0}(t) \right]_{t^{0}}, 
\end{equation}
where 
$[\cdot]_{t^{0}}$ means taking 
the constant term in $t$, and $p_{\Gamma^\infty}, p_{\Gamma_i^0}$ are projections from $\prod\limits_{\Gamma^0_i\in \mS_0}\bM_{\Gamma^0_i}^\sim(D) \times \bM^{\bullet}_{\Gamma^\infty}(X,D)$ to corresponding factors.

\begin{defn}\label{def-rel-cycle}
Define \emph{the relative Gromov--Witten cycle of the pair $(X,D)$ of topological type $\Gamma$} to be 
\[\mathfrak c_\Gamma(X/D) = \sum\limits_{\fG \in \mathcal B_\Gamma} \dfrac{1}{|Aut(\fG)|}(\mathfrak t_{\fG})_* (\iota^* C_{\fG} \cap [\bM_{\fG}]^{\vir}) \in A_*(\bM_{g,n+\rho}(X,\beta) \times_{X^\rho} D^{\rho}),\]
where $\iota$ is the vertical arrow in diagram \eqref{eqn:diag}.
\end{defn}

We define relative Gromov--Witten invariants (possibly with negative contact orders) by integrations against this cycle. 

Let 
\begin{align*}
    \begin{split}
        \ualpha = (\bar\psi^{a_1}\alpha_1,\ldots,\bar\psi^{a_n}\alpha_n) 
        &\in (\C[\bar\psi] \otimes H^*(X))^{n},\\ \uepsilon = (\bar\psi^{b_1}\epsilon_1,\ldots,\bar\psi^{b_{\rho}}\epsilon_\rho) &\in (\C[\bar\psi] \otimes H^*(D))^{\rho}.
    \end{split}
\end{align*}
There are evaluation maps from $\bM_\fG$ corresponding to interior markings and relative markings
\begin{align*}
\ev_X=(\ev_{X,1},\ldots,\ev_{X,n}):\bM_\fG&\rightarrow X^n, \\
\ev_D=(\ev_{D,1},\ldots,\ev_{D,\rho}):\bM_\fG&\rightarrow D^\rho.
\end{align*}
Recall we have the stabilization map $t_\fG:\bM_\fG\rightarrow \bM_{g,n+\rho}(X,\beta) \times_{X^\rho} D^{\rho}$. There are also evaluation maps
\begin{align*}
\overline{\ev}_X=(\overline\ev_{X,1},\ldots,\overline\ev_{X,n}):\bM_{g,n+\rho}(X,\beta) \times_{X^\rho} D^{\rho}&\rightarrow X^n, \\
\overline{\ev}_D=(\overline\ev_{D,1},\ldots,\overline\ev_{D,\rho}):\bM_{g,n+\rho}(X,\beta) \times_{X^\rho} D^{\rho}&\rightarrow X^\rho,
\end{align*}
such that 
\[
\overline{\ev}_X\circ t_\fG=\ev_X, \quad \overline{\ev}_D\circ t_\fG=\ev_D.
\]
\begin{defn}\label{rel-inv-neg}
The {\it relative Gromov--Witten invariant of topological type $\Gamma$ with insertions $\uepsilon,\ualpha$} is
\[
\langle \uepsilon \mid \ualpha \rangle_{\Gamma}^{(X,D)} =  \displaystyle\int_{\mathfrak c_\Gamma(X/D)} \prod\limits_{j=1}^{\rho} \bar{\psi}_{D,j}^{b_j}(\overline{\ev}_{D,j})^*\epsilon_j\prod\limits_{i=1}^n \bar{\psi}_{X,i}^{a_i}\overline{\ev}_{X,i}^*\alpha_i,
\]
where $\bar\psi_{D,j}, \bar\psi_{X,i}$ are pull-backs of psi-classes on $\bM_{g,n+\rho}(X,\beta)$ corresponding to relative and interior markings.
\end{defn}

\subsection{Two definitions coincide}\label{sec:match}

Recall that $\rho_-$, defined in equation (\ref{rho-neg}), is the number of large ages when $r$ is sufficiently large. In this section, we show that orbifold invariants of the root stack $X_{D,r}$ are Laurent polynomials in $r$ for $r$ sufficiently large. Furthermore, the lower bound for the degree of the polynomial is $-\rho_-$ and the coefficient of the lowest degree term of $r$ is exactly the relative invariants with $\rho_-$ negative relative markings (see Definition \ref{rel-inv-neg}). This is a generalization of the results of \cites{ACW, TY, TY18, FWY}. 

\begin{thm}\label{thm:rel-orb}
Consider the forgetful map
\[
\tau: \bM_\Gamma(X_{D,r})\rightarrow \bM_{g,n+\rho}(X,\beta) \times_{X^\rho} D^{\rho},
\] then
\[
r^{\rho_-}\tau_*\left([\bM_\Gamma(X_{D,r})]^{\on{vir}}\right)
\]
is a polynomial in $r$ for a sufficiently large $r$. 
Furthermore, we have the following relation for the cycle classes
\[
\mathop{\mathrm{lim}}_{r\rightarrow \infty} \left[r^{\rho_-}\tau_*\left([\bM_\Gamma(X_{D,r})]^{\on{vir}}\right)\right]_{r^0} = \sum\limits_{\fG \in \mathcal B_\Gamma} \dfrac{1}{|Aut(\fG)|}(\mathfrak t_{\fG})_* \left({\iota}^* C_{\fG} \cap [\bM_{\fG}]^{\on{vir}}\right).
\]
In other words, the two definitions of the relative Gromov--Witten cycles coincide.
\end{thm}
\begin{proof}
Similar to the proof of \cite{FWY}*{Theorem 6.1}, we prove Theorem \ref{thm:rel-orb} as follows. 

The first step is to consider the degeneration of $X_{D,r}$ to the normal cone and apply the degeneration formula to the cycle $[\bM_\Gamma(X_{D,r})]^{\vir}$. The degeneration formula is simply the higher genus version of the degeneration formula in \cite{FWY}*{(6.3)}. Hence, $[\bM_{\Gamma}(X_{D,r})]^{\vir}$ can be written in terms of suitable pushforward of $[\bM^{\bullet}_{\Gamma_1}(X,D)]^{\vir}$ and $[\bM^{\bullet}_{\Gamma_2}(P_{D_0,r},D_\infty)]^{\vir}$, where $P_{D_0,r}$ is the $r$th root stack of $P:=\Pp_D(L\oplus \sO)$ along $D_0$.

The second step is to compute $[\bM_{\Gamma}^{\bullet}(P_{D_0,r}, D_\infty)]^{\on{vir}}$ via virtual localization.
The virtual localization formula is as follows:

\begin{align}
&[\bM_{\Gamma}^{\bullet}(P_{D_0,r}, D_\infty)]^{\on{vir}}=\\
\notag & \qquad \sum_{\fG}\frac{1}{|\on{Aut}(\fG)|} \cdot\iota_*\left( \frac{[\bF_{\fG}]^{\on{vir}}}{e(\on{Norm}^{\on{vir}})}\right),
\end{align}
where $\bF_{\fG}$ is a fix locus indexed by the localization bipartite graph $\fG$ (see \cite{FWY}, Definition 4.10), the inverse of the virtual normal bundle $\frac{1}{e(\on{Norm}^{\on{vir}})}$ is a product of the following factors:

\begin{itemize}
\item for each stable vertex $v_i$ of $\Gamma^0_i$ over the zero section, we write $E_i$ for the set of edges associated to the vertex $v_i$ and write $\rho_{-}(i)$ for the number of $\infty$-roots of marking type (corresponding to large age markings) associated to $v_i$. The localization contribution is 
\begin{align}
\notag C_0(v_i)=&\left(\prod_{e\in E_i}\frac{rd_e}{t+\on{ev}_e^*c_1(L)-d_e\bar{\psi}_{e}}\right)\times\\
&\left(\sum_{j=0}^{\infty}(t/r)^{\g(i)-1+|E_i|-j+\rho_-(i)}c_j(-R^*\pi_*\mathcal L_r)\right),
\end{align}
where 
\[
\pi: \mathcal C_{\Gamma^0_i}(\cD_0)\rightarrow \bM_{\Gamma^0_i}(\cD_0)
\]
 is the universal curve, 
\[
f: \bM_{\Gamma^0_i}(\cD_0)\rightarrow \cD_0
\]
is the universal map and $\mathcal L_r=f^*L_r$ is the pullback of the universal $r$th root $L_r$ over the gerbe $\cD_0$.
The virtual rank of $-R^*\pi_*\mathcal L_r\in K^0(\bM_{\Gamma^0_i}(\cD_0))$ is $\g(i)+|E_i|+\rho_{-}(i)-1$.
As for each unstable vertex $v_i$, we have $C_0(v_i)=1$.

\item if the target expands over the infinity section, there is
a factor
\begin{align}
C_\infty=\frac{\prod_{e\in E(\fG)}d_e}{-t-\Psi_0}.
\end{align}
\end{itemize}

We consider the pushforward 
\[\tau:\bM_{\Gamma}^{\bullet}(P_{D_0,r}, D_\infty)\rightarrow  \bM_{\Gamma'}^\bullet(D)\]
where $\Gamma'$ is a topological type obtained from $\Gamma$. By slightly abuse of notation, we also use $\tau$ to denote the restriction of $\tau$ to each fixed locus.

To show that 
\[
r^{\rho_-}\tau_*\left([\bM_{\Gamma}^{\bullet}(P_{D_0,r}, D_\infty)]^{\on{vir}}\right)
\]
is a polynomial in $r$. We only need to show that for each $v_i$, 
\begin{equation}\label{contri-0}
\tau_*\left(r^{\rho_-(i)}C_0(v_i)\cap [\bM_{\Gamma^0_i}(\cD_0)]^{\on{vir}}\right)
\end{equation}
is a polynomial in $r$. We may rewrite \eqref{contri-0} as 
\[\prod_{e\in E_i}\frac{d_e}{t+(\on{ev}_e^*c_1(L)-d_e\psi_{e})}\left(\sum_{j=0}^{\infty}\hat{c}_j r^{\g(i)-j}t^{\g(i)-j+\rho_\infty(i)-1}\right)\]
where 
\[
\hat{c}_j=r^{2j-2\g(i)+1}\tau_*\left( c_j(-R^*\pi_*\mathcal L_r)\cap [\bM_{\Gamma^0_i}(\cD_0)]^{\on{vir}}\right).
\] 
By \cite{JPPZ18}, for each $j\geq 0$, the class $\hat{c}_j$ is a polynomial in $r$ when $r$ is sufficiently large. Furthermore, a refined polynomiality is proved in Corollary \ref{cor:cycrel} which states that $\hat{c}_j r^{\g(i)-j}$ is a polynomial in $r$, for each $j\geq 0$. Therefore, $r^{\rho_-}\tau_*\left([\bM_{\Gamma}^{\bullet}(P_{D_0,r}, D_\infty)]^{\on{vir}}\right)$ is a polynomial in $r$. We note that for $j<\g(i)$, the constant term of $\hat{c}_j r^{\g(i)-j}$ actually vanishes.

It remains to take the coefficient of $t^0r^{-\rho_-}$ in the class $\tau_*\left([\bM_{\Gamma}^{\bullet}(P_{D_0,r}, D_\infty)]^{\on{vir}}\right)$.  Then Theorem \ref{thm:rel-orb} follows from Corollary \ref{cor:HHI} by equating the constant terms of Hurwitz--Hodge cycles with rubber cycles.

\end{proof}

\subsection{Relative theory forms a partial CohFT}\label{sec:pcohft}
As a consequence of Theorem \ref{thm:rel-orb}, we will show that the all genera relative theory (possibly with negative contact) forms a partial cohomological field theory (CohFT) in the sense of \cite {LRZ}*{Definition 2.7}. 

We begin with a brief review of the axioms of CohFT. Let $\bM_{g,m}$ be the moduli space of of genus $g$, $m$-pointed stable curves. We assume that $2g-2+m>0$. There are several canonical morphisms between the $\bM_{g,m}$. 
\begin{itemize}
    \item Forgetful morphism
    \[
    \pi:\bM_{g,m+1}\rightarrow \bM_{g,m}
    \]
    obtained by forgetting the last marking.
    \item Gluing the loop
    \[
    \rho_l: \bM_{g,m+2}\rightarrow \bM_{g+1,m}
    \]
    obtained by identifying the last two markings of the $(m+2)$-pointed, genus $g$ curves.
    \item Gluing the tree
    \[
    \rho_t:\bM_{g_1,m_1+1}\times \bM_{g_2,m_2+1}\rightarrow \bM_{g_1+g_2,m_1+m_2}
    \]
    obtained by identifying the last markings of separate pointed curves.
\end{itemize}

Let $H$ be a graded vector space with a non-degenerate pairing $\langle ,\rangle$ and a distinguished element $1\in H$. Given a basis $\{e_i\}$, let $\eta_{jk}=\langle e_j,e_k\rangle$ and $(\eta^{jk})=(\eta_{jk})^{-1}$.

A CohFT is a collection of homomorphisms
\[
\Omega_{g,m}: H^{\otimes m}\rightarrow H^*(\bM_{g,m},\mathbb Q)
\]
satisfying the following axioms:
\begin{itemize}
    \item The element $\Omega_{g,m}$ is invariant under the natural action of symmetric group $S_m$.
    \item For all $a_i\in H$, $\Omega_{g,m}$ satisfies
    \[
    \Omega_{g,m+1}(a_1,\ldots,a_m,1)=\pi^*\Omega_{g,m}(a_1,\ldots,a_m).
    \]
    \item The splitting axiom:
    \begin{align*}
    &\rho^*_t\Omega_{g_1+g_2,m_1+m_2}(a_1,\ldots,a_{m_1+m_2})=\\
    &\sum_{j,k}\eta^{jk}\Omega_{g_1,m_1}(a_1,\ldots,a_{m_1},e_j)\otimes \Omega_{g_2,m_2}(a_{m_1+1},\ldots,a_{m_1+m_2},e_k),
    \end{align*}
    for all $a_i\in H$.
    \item The loop axiom:
    \[
    \rho_l^*\Omega_{g+1,m}(a_1,\ldots,a_m)=\sum_{j,k}\eta^{jk}\Omega_{g,m+2}(a_1,\ldots,a_m,e_j,e_k),
    \]
    for all $a_i\in H$. In addition, the equality
    \[
    \Omega_{0,3}(v_1,v_2,1)=\langle v_1,v_2\rangle
    \]
    holds for all $v_1,v_2\in H$.
\end{itemize}

\begin{defn}[\cite{LRZ}, Definition 2.7]
If the collection $\{\Omega_{g,m}\}$ satisfies all the axioms except for the loop axiom, we call it a partial CohFT.
\end{defn}

Following \cite{FWY}, the ring of insertions for relative Gromov--Witten theory is defined to be
\[
\HH=\bigoplus\limits_{i\in\Z}\HH_i,
\]
where $\HH_0=H^*(X)$ and $\HH_i=H^*(D)$ if $i\in \Z - \{0\}$.
For an element $\alpha\in \HH_i$, we write $[\alpha]_i$ for its embedding in $\HH$. 

The pairing on $\HH$ is defined as follows.
\begin{equation}\label{eqn:pairing}
\begin{split}
([\alpha]_i,[\beta]_j) = 
\begin{cases}
0, &\text{if } i+j\neq 0,\\
\int_X \alpha\cup\beta, &\text{if } i=j=0, \\
\int_D \alpha\cup\beta, &\text{if } i+j=0, i,j\neq 0.
\end{cases}
\end{split}
\end{equation}
In the next definition, we need to use the forgetful map 
\[
\pi:\bM_{g,m}(X,\beta) \times_{X^\rho} D^\rho \rightarrow \bM_{g,m},
\]
where a list of contact orders will be implied in the context, and the fiber product remembers the $\rho$ markings that correspond to relative markings. 

For a sufficiently large integer $r$, an element $[\alpha]_i$ naturally corresponds to a cohomology class $\alpha\in H^*(\underline{I}X_{D,r})$ that lies in either a component with age $i/r$ (if $i\geq0$), or a component with age $(r+i)/r$ (if $i<0$). To simplify the notation in the next definition, we later refer to a general element of $\HH$ by simply writing $[\alpha]\in \HH$ without a subscript. 
\begin{defn}
Given elements $[\alpha_1],\ldots,[\alpha_m]\in \HH$, the relative Gromov--Witten class is defined as
\[
\Omega^{(X,D)}_{g,m,\beta}([\alpha_1],\ldots,[\alpha_m])=\pi_*\left(\prod_{i=1}^m\overline{ev}_i^*(\alpha_i)\cap\mathfrak c_\Gamma(X/D)\right)\in H^*(\bM_{g,m},\mathbb Q),
\]
where the topological type $\Gamma$ is determined by $g,m,\beta$ and the insertions $[\alpha_1],\ldots,[\alpha_m]\in \HH$. We then define the class
\[
\Omega^{(X,D)}_{g,m}([\alpha_1],\ldots,[\alpha_m])=\sum_{\beta\in H_2(\mathcal X,\mathbb Q)}\Omega^{(X,D)}_{g,m,\beta}([\alpha_1],\ldots,[\alpha_m])q^\beta
\]
\end{defn}

\begin{thm}
$\Omega^{(X,D)}_{g,m}$ forms a partial CohFT.
\end{thm}

\begin{proof}
It is easy to see that our relative Gromov--Witten theory satisfies the first two axioms of the CohFT. The splitting axiom for relative Gromov--Witten theory can be proved by using the splitting axiom for the corresponding orbifold Gromov--Witten theory. We write $\Omega^{X_{D,r}}_{g,m,\beta}(\alpha_1,\ldots,\alpha_m)$ for the orbifold Gromov--Witten class of the root stack $X_{D,r}$. For $r$ sufficiently large, we have
\begin{align*}
    &\rho^*_t\Omega^{(X,D)}_{g_1+g_2,m_1+m_2}([\alpha_1],\ldots,[\alpha_{m_1+m_2}])\\
    =&\rho^*_t\left[r^{\rho_-}\Omega^{X_{D,r}}_{g_1+g_2,m_1+m_2}(\alpha_1,\ldots,\alpha_{m_1+m_2})\right]_{r^0}\\
    =&\left[r^{\rho_-}\sum_{j,k}\eta^{jk}\Omega_{g_1,m_1}^{X_{D,r}}(\alpha_1,\ldots,\alpha_{m_1},e_j)\otimes \Omega_{g_2,m_2}^{X_{D,r}}(\alpha_{m_1+1},\ldots,\alpha_{m_1+m_2},e_k)\right]_{r^0}\\
    =&\sum_{j,k}\eta^{jk}\left[r^{\rho_{1-}}\Omega_{g_1,m_1}^{X_{D,r}}(\alpha_1,\ldots,\alpha_{m_1},e_j)\right]_{r^0}\otimes \left[r^{\rho_{2-}}\Omega_{g_2,m_2}^{X_{D,r}}(\alpha_{m_1+1},\ldots,\alpha_{m_1+m_2},e_k)\right]_{r^0}\\
    =&\sum_{j,k}\eta^{jk}\Omega_{g_1,m_1}^{(X,D)}([\alpha_1],\ldots,[\alpha_{m_1}],[e_j])\otimes \Omega_{g_2,m_2}^{(X,D)}([\alpha_{m_1+1}],\ldots,[\alpha_{m_1+m_2}],[e_k])
\end{align*}
Some explanations and remarks about the equalities are as follows: 
\begin{enumerate}
    \item $\rho_-, \rho_{1-}, \rho_{2-}$ are the numbers of negative markings of the corresponding topological types.
    \item The contact orders (and ages) of $e_j$ and $e_k$ are already determined by the insertions $\alpha_i$.
      \item In the case when the new marking corresponding to $e_j$ is a large age marking, the class $e_j\in H^*(D_r)$ is identified with $r e_j\in H^*(D)$ as the orbifold pairing for the $\mu_r$-gerbe $D_r$ requires an extra factor of $r$. Similarly, for $e_k$. Therefore the fourth line is valid.
    \item The summation on the fifth line is a finite sum.
    \item We use the same notation $\eta^{jk}$ for both orbifold Gromov--Witten classes and relative Gromov--Witten classes, although $\eta^{jk}$ actually lives in different spaces for these two cases. 
\end{enumerate}
 \end{proof}

Next, we will give an example which shows that our relative theory do not always satisfy the loop axiom.
\begin{ex}
Let $X=\mathbb{P}^1$ and $D=pt$ be a point in $X$. Let $1$ to be the identity in the cohomology ring and $\omega$ is the point class in $X$. Let us consider $\Omega^{(X,D)}_{1,1,0}([1]_0)$. In this case, $\bM_{1,1}(X,0)=\bM_{1,1}\times X$ and 
\[\mathfrak c_\Gamma(X/D)=[\bM_{\Gamma}(X,D)]^{\vir}=c_1(\mathbb{E}^{\vee}\boxtimes T_X(-\mathrm{log} D))\cap [\bM_{1,1}\times X] \]
where $\mathbb{E}$ is the Hodge bundle. So by definition, it is easy to compute that $\Omega^{(X,D)}_{1,1,0}([1]_0)=1\in H^0(\bM_{1,1})$. Now if the loop axiom holds, it must equal to 
\[2\Omega^{(X,D)}_{0,3,0}([1]_0,[1]_0,[\omega]_0)+\sum_{i\in \mathbb{Z}^*}\Omega^{(X,D)}_{0,3,0}([1]_0,[1]_i,[1]_{-i}).\]
But since $\mathfrak c_\Gamma(X/D)$ in $\Omega^{(X,D)}_{0,3,0}([1]_0,[1]_i,[1]_{-i})$ contains only one negative marking, Example 5.5 in \cite{FWY} implies that 
\[\Omega^{(X,D)}_{0,3,0}([1]_0,[1]_i,[1]_{-i})=1\in H^0(\bM_{0,3}),\quad \forall\,i\in \mathbb{Z}^*.\]
Now $\sum_{i\in \mathbb{Z}^*}\Omega^{(X,D)}_{0,3,0}([1]_0,[1]_i,[1]_{-i})$ is an infinite sum which is not convergent! So in this case, the loop axiom does not hold.
\end{ex}

At this point, we do not know how to find a replacement of the loop axiom.
 


\section{Hurwitz--Hodge classes and DR-cycles}\label{sec:HH-DR}
In this section, we extend the trick in \cite{FWY}*{Appendix} to higher genera. We first summarize the main result of this section.

Recall that $P_{D_0,r}$ is the $r$th root stack of $P:=\Pp_D(L\oplus \sO)$, and $\cD_0, \cD_\infty$ are two invariant substacks. $\cD_0$ is isomorphic to $\sqrt[r]{D/L}$ and $\cD_\infty$ is isomorphic to $D$.
Let $\bM_{\Gamma}^\sim(D)$ be the moduli of relative stable maps to rubber
targets over $D$. Suppose that $\Gamma$ imposes weights $\vec{\mu}=(\mu_1,\ldots,\mu_{\rho_0})$ on $0$-roots
, $\vec{\nu}=(-\nu_1,\ldots,-\nu_{\rho_\infty})$ on $\infty$-roots, and denote the number of legs by $n$. In other words, we have tangency conditions $\vec{\mu}$ on $0$-side, $-\vec{\nu}$ on $\infty$-side and $n$ interior markings. For orbifold theory, we define
a vector of ages
\[
\vec{a}=(
(r-\nu_1)/r,\ldots,(r-\nu_{\rho_\infty})/r,\mu_1/r,\ldots,\mu_{\rho_0}/r,\underbrace{0,\ldots,0}_n
 ).
\]
There are the following forgetful
maps
\begin{align*}
  \begin{split}
  \tau_1:&\bM_{g,\vec{a}}(\cD_0,\beta)\rightarrow \bM_{g,n+\rho_0+\rho_\infty}(D,\beta),\\ 
  \tau_2:&\bM_{\Gamma}^\sim(D)\rightarrow \bM_{g,n+\rho_0+\rho_\infty}(D,\beta).
  \end{split}
\end{align*}
Here under $\tau_2$, $\infty$-roots are identified as markings $1,\ldots,\rho_\infty$, $0$-roots as markings $\rho_\infty+1,\ldots,\rho_\infty+\rho_0$, and legs as markings $\rho_\infty+\rho_0+1,\ldots,\rho_\infty+\rho_0+n$.

Recall $\Psi_\infty$ is the divisor corresponding to the cotangent line bundle determined by the relative divisor on $\infty$ side.
On the root gerbe $\cD_0$, there is a universal line bundle $L_r$. Let $\pi:\mathcal
C\rightarrow \bM_{g,\vec{a}}(\cD_0,\beta)$ be the universal curve and $f:\mathcal C\rightarrow
\cD_0$ the map to the target. Let $\mathcal L_r = f^*L_r$ and 
\[-R^*\pi_*\mathcal L_r:=R^1\pi_*\mathcal L_r-R^0\pi_*\mathcal L_r \in K^0(\bM_{g,\vec{a}}(\cD_0,\beta)).\]

\begin{thm}\label{thm:appdx}
Let $t$ be a formal variable and $r\gg 1$. Under the above notation, the cycle 
\[(\tau_1)_*\left(\sum\limits_{i=0}^\infty \left(\dfrac{t}{r}\right)^{g-i-1}c_i(-R^*\pi_*\mathcal L_r) \cap [\bM_{g,\vec{a}}(\cD_0,\beta)]^{\on{vir}} \right)\]
is a polynomial in $r$. And
we have the following equality 
\begin{align*}
\begin{split}
&\left[  (\tau_1)_*\left(\sum\limits_{i=0}^\infty \left(\dfrac{t}{r}\right)^{g-i-1}c_i(-R^*\pi_*\mathcal L_r) \cap [\bM_{g,\vec{a}}(\cD_0,\beta)]^{\on{vir}} \right)  \right]_{r^0}  \\
=& {\prod\limits_{i=1}^{\rho_\infty}\left(1+\dfrac{\ev_i^*c_1(L)-\nu_i\psi_i}{t}\right)} \cap (\tau_2)_*\left(\dfrac{1}{t-\Psi_\infty} \cap [\bM_{\Gamma}^\sim(D)]^{\on{vir}}\right).
\end{split}  
\end{align*}
in $A_*(\bM_{g,n+\rho_0+\rho_\infty}(D,\beta))[t,t^{-1}]$ as Laurent polynomials in $t$
\end{thm}
This result has a few consequences. For example, the polynomiality implies that
\begin{cor}\label{cor:cycrel}
For $r\gg 1$, the cycle \[(\tau_1)_*\left(r^{i-g+1}c_i(-R^*\pi_*\mathcal L_r)\cap [\bM_{g,\vec{a}}(\cD_0,\beta)]^{\on{vir}}\right)\]
is a polynomial in $r$.
\end{cor}
In \cite{JPPZ18}, we already know that $r^{2i-2g+1}(\tau_1)_*c_i(-R^*\pi_*\mathcal L_r)$ is a polynomial in $r$ in Chow cohomology. When $i>g$, this corollary is an improved bound of the lowest degree after capping with the virtual class. On the other hand, we also conclude that those $i<g$ summands do not contribute to the left hand side of Theorem \ref{thm:appdx}. Precisely, we have
\begin{equation}\label{eqn:appvanish}
\left[  (\tau_1)_*\left(r^{i-g+1}c_i(-R^*\pi_*\mathcal L_r)\cap [\bM_{g,\vec{a}}(\cD_0,\beta)]^{\on{vir}} \right) \right]_{r^0} = 0
\end{equation}
for $i<g$.

As another interpretation of Theorem \ref{thm:appdx}, we have a precise formula of the leading term of $c_i(-R^*\pi_*\mathcal L_r)$. For $1\leq i\leq \rho_\infty$, we write 
\begin{equation}\label{eqn:p}
p_i=\nu_i\tau_2^*\psi_{i}-\ev_{i}^*c_1(L),
\end{equation}
and 
\[
\sigma_l=\sum\limits_{1\leq i_1<\ldots<i_l\leq\rho_\infty}p_{i_1}\ldots p_{i_l}.
\]
We set $\sigma_l=0$ if $l>\rho_{\infty}$.

\begin{cor}\label{cor:HHI}
Under the above notation, for any positive integer $k$, $r\gg 1$ and $i\geq g$, the following equality holds in $A_*(\bM_{g,n+\rho_0+\rho_\infty}(D,\beta))$.
\begin{align*}
  \begin{split}
    &\left[ (\tau_1)_*\left(r^{i-g+1} c_i(-R^*\pi_*\mathcal L_r)\cap [\bM_{g,\vec{a}}(\cD_0)]^{\on{vir}}\right) \right]_{r^0} \\
    =&(\tau_2)_*\left( (\Psi_\infty^{i-g}-\Psi_\infty^{i-g-1}\sigma_1+\ldots+(-1)^{i-g}\sigma_{i-g}) \cap [\bM_{\Gamma}^\sim(D)]^{\on{vir}}\right).
  \end{split}
\end{align*}
\end{cor}

Conversely, if one cares about DR-cycle, there is a formula for $(\tau_2)_*(\Psi_\infty^i\cap [\bM_{\Gamma}^\sim(D)]^{\vir})$.

\begin{cor}
The following equality holds in $A_*(\bM_{g,n+\rho_0+\rho_\infty}(D,\beta))$ for $r\gg 1$.
\begin{align*}
    &(\tau_2)_*\left(\Psi_\infty^i\cap [\bM_{\Gamma}^\sim(D)]^{\on{vir}}\right) \\
    =& (\tau_1)_*\left[ \sum\limits_{k_1,\ldots,k_{\rho_\infty}\geq 0} \left(\prod\limits_{j=1}^{\rho_\infty} (p_{j}')^{k_j}\right) r^{i+1-\sum k_j}c_{i-\sum k_j+g}(-R^*\pi_*\mathcal L_r) \cap [\bM_{g,\vec{a}}(\cD_0)]^{\on{vir}}\right]_{r^0},
\end{align*}
where $p_j'=\nu_j\tau_1^*\psi_{j}-\ev_{j}^*c_1(L)$.
\end{cor}
If we take $i=0$, the right-hand side only involves $c_{g-\sum k_j}(-R^*\pi_*\mathcal L_r)$. By \eqref{eqn:appvanish}, only the summand with $c_g(-R^*\pi_*\mathcal L_r)$ might have nontrivial contribution. This in fact recovers the result in \cite{JPPZ18}*{Section 3.5.4} equation (25).

\subsection{Families of local models}
The rest of the section provides a proof of Theorem \ref{thm:appdx}. The idea is similar to \cite{FWY}*{Appendix}, but a lot of adjustments need to be made for higher genera.

We first simplify the notation and match them with \cite{FWY}*{Appendix}. Let $\bM^{\rel}(P,D_0)$ be the moduli of relative stable maps (corresponding to our $\bM_\Gamma(P,D_0)$), and $\bM^{\text{orb}}(P_{D_0,r})$ be the moduli of orbifold stable maps (corresponding to our $\bM_{\Gamma}(P_{D_0,r})$). There is the following diagram.

\begin{equation*}
\xymatrix{
\bM^{\rel}(P,D_0) \ar[rd]^{\Psi} & & \bM^{\text{orb}}(P_{D_0,r}) \ar[ld]_{\Phi}\\
& \bM_{g,n+\rho}(P,\beta).   & 
}
\end{equation*}

The following result is a slightly stronger version of \cite{TY}.
\begin{lem}\label{lem:appnoneq}
For $r\gg1$, the cycle 
\[\Phi_*([\bM^{\text{orb}}(P_{D_0,r}))\in A_*(\bM_{g,n+\rho}(P,\beta))\]
is a polynomial in $r$ and we have the following identity:
\[
\Psi_*\left([\bM^{\on{rel}}(P,D_0)]^{\on{vir}}\right) = \bigg[ \Phi_*([\bM^{\text{orb}}(P_{D_0,r})]^{\on{vir}}) \bigg]_{r^0}.
\]
\end{lem}
Notice there is a $\GM$-action on $P_{D_0,r}$ which is compatible with the scaling action on fiber on $P$. What we are aiming for is in fact the following equivariant version of the lemma.
\begin{lem}\label{lem:appeq}
For $r\gg 1$, the cycle
\[\Phi_*\left([\bM^{\on{orb}}(P_{D_0,r})]^{\on{vir},\on{eq}}\right)\in A_*^{\C^*}(\bM_{g,n+\rho}(P,\beta))  \]
is a polynomial in $r$ and we have the following identity:
\[
\Psi_*\left([\bM^{\on{rel}}(P,D_0)]^{\on{vir},\on{eq}}\right) = \bigg[ \Phi_*\left([\bM^{\on{orb}}(P_{D_0,r})]^{\on{vir},\on{eq}}\right) \bigg]_{r^0}.
\]
\end{lem}

Unlike \cite{FWY}*{Appendix}, Lemma \ref{lem:appnoneq} does not imply Lemma \ref{lem:appeq} directly when $g>0$. Instead, we need to extend Lemma \ref{lem:appnoneq} to a family over a base. Let $\pi:E\rightarrow B$ be a smooth morphism between two smooth algebraic varieties. Furthermore, suppose that $E$ is also a $\C^*$-torsor over $B$. The previous diagram can be modified into a family over $B$.
\begin{equation*}
\xymatrix{
\bM^{\rel}(P\times_{\C^*}E,D_0\times_{\C^*}E) \ar[rd]^{\Psi_{E}} & & \bM^{\text{orb}}(P_{D_0,r}\times_{\C^*}E) \ar[ld]_{\Phi_{E}}\\
& \bM_{g,n+\rho}(P\times_{\C^*}E,\beta),  & 
}
\end{equation*}
where $P\times_{\C^*}E = (P\times E)/\C^*$ with $\C^*$ acts on both factors, and the curve class $\beta$ is a fiber class (projects to $0$ on $B$). Purely as moduli spaces, these moduli spaces can be explicitly described as follows.
\[
\bM^{\rel}(P\times_{\C^*}E,D_0\times_{\C^*}E) \cong \bM^{\rel}(P,D_0)\times_{\C^*}E,
\]
\[
\bM^{\text{orb}}(P_{D_0,r}\times_{\C^*}E) \cong \bM^{\text{orb}}(P_{D_0,r})\times_{\C^*}E,
\]
\[
\bM_{g,n+\rho}(P\times_{\C^*}E,\beta) \cong \bM_{g,n+\rho}(P,\beta) \times_{\C^*} E.
\]
But they admit relative perfect obstruction theories over $B$ and thus forming virtual cycles relative to the base $B$ (denoted by $[\ldots]^{\vir_{\pi}}$).

\begin{lem}\label{lem:appfamily}
For $r\gg 1$, the cycle 
\[(\Phi_E)_*\left([\bM^{\on{orb}}(P_{D_0,r})\times_{\C^*}E]^{\on{vir}_{\pi}}\right)\in A_*(\bM_{g,n+\rho}(P\times_{\C^*}E,\beta))\]
is a polynomial in $r$ and we have the following identity:
\[
(\Psi_E)_*\left([\bM^{\on{rel}}(P\times_{\C^*}E,D_0\times_{\C^*}E)]^{\on{vir}_{\pi}}\right) = \bigg[ (\Phi_E)_*\left([\bM^{\on{orb}}(P_{D_0,r})\times_{\C^*}E]^{\on{vir}_{\pi}}\right) \bigg]_{r^0}.
\]
\end{lem}
The proof of all three lemmas will be given in the next two subsections. The idea is the following. First of all, since the method in \cite{TY} is localization formula which holds on the Chow level, one can adapt their proof into Lemma \ref{lem:appnoneq} without too much trouble. On the other hand, the fibration $P\times_{\C^*}E$ admits a fiberwise $\C^*$-action by letting $\C^*$ act on the first factor. Localization analysis thus applies to Lemma \ref{lem:appfamily}, and every procedure in the proof of Lemma \ref{lem:appnoneq} works out in parallel. In the end, Lemma \ref{lem:appfamily} implies Lemma \ref{lem:appeq} because equivariant theory is a limit of nonequivariant theories over families. Or in a slightly different language, Lemma \ref{lem:appeq} is Lemma \ref{lem:appfamily} with $E=E\C^*$, $B=B\C^*$ (the classifying space of $\C^*$).

\subsection{Localization on (families of) local models}
The goal is to adjust the arguments in \cite{TY} into cycle forms like Lemma \ref{lem:appnoneq}, and check whether it works on families of local models over a base. We outline the adjusted argument but omit details of some standard procedures (some details see \cites{TY18,TY,JPPZ,JPPZ18}). This adjusted argument in fact simplifies the argument of \cite{TY}.

\subsubsection*{Step $1$, degenerate the target.} We are comparing $\bM^{\rel}(P,D_0)$ and $\bM^{\text{orb}}(P_{D_0,r})$. However, for computational purposes, we apply deformation to the normal cones to divisors $D_0, \cD_0$, respectively. $P$ degenerates into two copies of $P$ with $D_0$ and $D_\infty$ glued together. By degeneration formula, $[\bM^{\rel}(P,D_0)]^{\vir}$ can be written in terms of suitable pushforward of $[\bM^{\bullet\rel}(P,D_0)]^{\vir}$ and $[\bM^{\bullet\rel}(P,D_0\cup D_\infty)]^{\vir}$. On the other hand, $[\bM^{\text{orb}}(P_{D_0,r})]^{\vir}$ can be written in terms of suitable pushforward of $[\bM^{\bullet\rel}(P,D_0)]^{\vir}$ and $[\bM^{\bullet\text{orb,rel}}(P_{D_0,r},D_\infty)]^{\vir}$. Here $\bM^{\bullet\text{orb,rel}}(P_{D_0,r},D_\infty)$ has orbifold structures along $D_0$ as well as tangency conditions along $D_\infty$ (called \emph{orbifold/relative} theory of $P_{D_0,r}$), and $\bullet$ means possibly disconnected domains. We omit the details of degeneration formulas. But schematically, the two formulas can be summarized as follows.
\[
  [\bM^{\rel}(P,D_0)]^{\vir}\leadsto [\bM^{\bullet\rel}(P,D_0)]^{\vir}\text{ and }[\bM^{\bullet\rel}(P,D_0\cup D_\infty)]^{\vir} .
\]
 \[
[\bM^{\text{orb}}(P_{D_0,r})]^{\vir} \leadsto [\bM^{\bullet\rel}(P,D_0)]^{\vir}\text{ and }[\bM^{\bullet\text{orb,rel}}(P_{D_0,r},D_\infty)]^{\vir}.
 \]

The two sums in the degeneration formula range over the same set of intersection profiles along $D$, and thus can be matched term-by-term. In the end, it suffices to compare the pushforward of virtual cycles $[\bM^{\bullet\rel}(P,D_0\cup D_\infty)]^{\vir}, [\bM^{\bullet\text{orb,rel}}(P_{D_0,r},D_\infty)]^{\vir}$ with matching profiles along $D_0, D_\infty$. Note that {\bf this step works on families of $P$ over a base $B$ without problems.}

\subsubsection*{Step $2$, localization formula.} We now focus on the following diagram\footnote{One potential confusing point is that we choose to pushforward to $\bM^{\bullet}_{g,n+\rho_0+\rho_\infty}(D,\beta)$ instead of $\bM^{\bullet}_{g,n+\rho_0+\rho_\infty}(P,\beta)$. In the deformation to the normal cone, the pairs $(P,D_0\cup D_\infty)$ and $(P_{D_0,r},D_\infty)$ are in fact coming out of exceptional divisors. When writing degeneration formula, the corresponding cycles should be projected to moduli of stable maps to $D$ first and then embedded into $P$.}:
\begin{equation*}
\xymatrix{
\bM^{\bullet\rel}(P,D_0\cup D_\infty) \ar[rd]^{\Psi'} & & \bM^{\bullet\text{orb,rel}}(P_{D_0,r},D_\infty) \ar[ld]_{\Phi'}\\
& \bM^{\bullet}_{g,n+\rho_0+\rho_\infty}(D,\beta) & 
}
\end{equation*}
$\bM^{\bullet}_{g,n+\rho_0+\rho_\infty}(D,\beta)$ can be written as a product of moduli spaces with connected domain curves. In the product, it is possible that some of the factors are unstable, that is, genus zero, two markings and curve class zero. Then we simply set it to be $D$. We call $\bM^{\bullet}_{g,n+\rho_0+\rho_\infty}(D,\beta)$   \emph{unstable}, if all the factors are unstable.

We aim to prove the following.
\begin{lem}\label{lem:appdxloc}
For $r\gg 1$, the cycle 
\[\Phi'_*\left([\bM^{\bullet\on{orb,rel}}(P_{D_0,r},D_\infty)]^{\on{vir}}\right)\in A_*(\bM^{\bullet}_{g,n+\rho_0+\rho_\infty}(D,\beta)) \]
is a polynomial in $r$ and we have the following identity:
\[
\Psi'_*\left([\bM^{\bullet\on{rel}}(P,D_0\cup D_\infty)]^{\on{vir}}\right) = \bigg[  \Phi'_*\left([\bM^{\bullet\on{orb,rel}}(P_{D_0,r},D_\infty)]^{\on{vir}}\right)  \bigg]_{r^0}.
\]
In fact, it is nonzero if and only if $\bM^{\bullet}_{g,n+\rho_0+\rho_\infty}(D,\beta)$ is unstable.
\end{lem}
\begin{proof}[Proof of Lemma \ref{lem:appdxloc}]
Since $r$ is sufficiently large, localization formulas on both $\bM^{\bullet\rel}(P,D_0\cup D_\infty)$ and $\bM^{\bullet\text{orb,rel}}(P_{D_0,r},D_\infty)$ do not have edge contributions. Vertex contributions are compared as follows.
\begin{itemize}
\item Suppose that there is a stable vertex (target expands) $v$ in the localization of $\bM^{\bullet\rel}(P,D_0\cup D_\infty)$. If $v$ is a vertex over the zero section, it will contribute
\[
\dfrac{\prod_{e\in E(v)}d_e}{t-\psi}.
\]
If $v$ is a vertex over the $\infty$ section, it will contribute
\[\dfrac{\prod_{e\in E(v)}d_e}{-t-\psi}.\]
Here $e$ ranges over $E(v)$, the set of edges connecting to the vertex; and $\psi$ is the cotangent line class on the rubber at the corresponding boundary divisor; $d_e$ are degrees of the corresponding edges.

\item Suppose that there is a stable vertex $v$ in the localization of $\bM^{\bullet\text{orb,rel}}(P_{D_0,r},D_\infty)$. If $v$ is a vertex over the gerbe $\cD_0$, the vertex contribution is
\[
\prod\limits_{e\in E(v)} \dfrac{d_e}{1+\frac{\ev_e^*(L)-d_e\bar\psi_e}{t}}\sum\limits_{i\geq 0} \left( \dfrac{t}{r} \right)^{\g(v)-i-1} c_i(-R^*\pi_*\mathcal L_r).
\]
If $v$ is a vertex over $D_\infty$, the contribution is 
\[\dfrac{\prod_{e\in E(v)}d_e}{-t-\psi}.\]
\end{itemize}
As one can check, stable vertices in the relative theory of $(P,D_0\cup D_\infty)$ contribute $O(1/t)$ which do not affect the non-equivariant limits. In orbifold/relative theory of $(P_{D_0,r},D_\infty)$, we consider stable vertices over $\cD_0$. Let $v_j$ be such a vertex. We consider the natural forgetful map associated to $v_j$
\[\tau: \bM_{\Gamma^0_j}(\cD_0)\rightarrow \bM_{\g(j),n(j)+\rho(j)}(D,b(j)).\]
By \cite{JPPZ18}*{Corollary 10}, $r^{2i-2g+1}\tau_*c_i(-R^*\pi_*\mathcal L_r)$ is a polynomial in $r$. And by the compatibility of the virtual cycles (see \cite{AJT}*{Theorem 6.8}),
\[\tau_*\left([\bM_{\Gamma^0_j}(\cD_0)]^{\vir}\right)=r^{2\g(j)-1}[\bM_{\g(j),n(j)+\rho(j)}(D,b(j))]^{\vir}.\]
So one may check that for $k\geq 0$, the $r^{-k}$ coefficients of
\[\tau_*\left(\prod\limits_{e\in E(v)} \dfrac{d_e}{1+\frac{\ev_e^*(L)-d_e\bar\psi_e}{t}}\sum\limits_{i\geq 0} \left( \dfrac{t}{r} \right)^{\g(v)-i-1} c_i(-R^*\pi_*\mathcal L_r)\cap [\bM_{\Gamma^0_j}(\cD_0)]^{\vir}\right)\]
only involve of negative powers of $t$. The contribution over $D_{\infty}$ also contain negative powers of $t$. So after taking a non-equivariant limit, we may conclude that \[\Phi'_*\left([\bM^{\bullet\on{orb,rel}}(P_{D_0,r},D_\infty)]^{\on{vir}}\right)\]
is a polynomial in $r$.

Moreover, if $\bM^{\bullet}_{g,n+\rho_0+\rho_\infty}(D,\beta)$ is not unstable (see the paragraph before Lemma \ref{lem:appdxloc}), then each localization contribution on both moduli spaces must involve of stable vertices. So 
\[
\Psi'_*\left([\bM^{\bullet\on{rel}}(P,D_0\cup D_\infty)]^{\on{vir}}\right) = \bigg[  \Phi'_*\left([\bM^{\bullet\on{orb,rel}}(P_{D_0,r},D_\infty)]^{\on{vir}}\right)  \bigg]_{r^0}=0.
\]
If $\bM^{\bullet}_{g,n+\rho_0+\rho_\infty}(D,\beta)$ is unstable, then those contact orders satisfy $\mu_i=\nu_i$ for $0\leq i\leq \rho_0$ and
\begin{align*}
\Psi'_*\left([\bM^{\bullet\on{rel}}(P,D_0\cup D_\infty)]^{\on{vir}}\right)&=\Phi'_*\left([\bM^{\bullet\on{orb,rel}}(P_{D_0,r},D_\infty)]^{\on{vir}}\right)\\
&=\frac{1}{\prod_{i=1}^{\rho_0} \mu_i}[D\times D\cdots\times D].
\end{align*}
\end{proof}

Recall that $E\rightarrow B$ is a $\C^*$-torsor over a smooth base $B$. When we replace $P$ by $P\times_{\C^*} E$, the boundary divisors $D_0,D_\infty$ are changed into $D_0\times_{\C^*} E$ and $D_\infty\times_{\C^*}E$. Since $D_0,D_\infty$ are fixed under $\C^*$,
\[
D_0\times_{\C^*} E\cong D_\infty\times_{\C^*} E \cong D\times B.
\]
There is an extra $\C^*$-action on $P\times_{\C^*} E$ by only acting on the first factor. It's easy to check that fixed loci are $D_0\times_{\C^*} E$ and $D_\infty\times_{\C^*} E$. The use of relative perfect obstruction theory over $B$ causes vertices to correspond to moduli equipped with $\pi$-relative virtual classes including $[\bM(D_0\times_{\C^*}E)^\sim]^{\vir_\pi}, [\bM(\cD_0\times_{\C^*}E)]^{\vir_\pi}$, etc. But this has no effect on our analysis of ruling out stable vertices. The computation of edge contributions remains the same (tangent bundle of $B$ does not contribute to the moving part).

The only difference that might make one slightly cautious is whether we can apply \cite{JPPZ18}*{Corollary 10} (polynomiality) and the compatibility of the virtual cycles still holds.
In family case, this polynomiality is under the following set-up.
Let $L$ be a line bundle over $D\times B$ which is the pullback of a line bundle $L_D$ over $D$. Let $\sqrt[r]{D\times B/L}$ be the $r$th root gerbe of $D\times B$. Let $\mathcal L_r$ be the universal $r$th root over the universal curve on $\bM_{g,\vec{a}}(\sqrt[r]{D\times B/L},\beta))$ where $\vec{a}$ is a fixed age vector and $\beta$ is a fixed fiber class (pushforward to $0$ in $H_2(B,\Z)$). Let $\tau_1$ be the forgetful map 
\[
\tau_1:\bM_{g,\vec{a}}(\sqrt[r]{D\times B/L},\beta))\rightarrow \bM_{g,n+\rho_0+\rho_\infty}(D\times B,\beta).
\]
Note that 
\[\sqrt[r]{D\times B/L} \cong \sqrt[r]{D/L_D}\times B.\]
We want $r^{2i-2g+1}(\tau_1)_*c_i(-R^*\pi_*\mathcal L_r)$
to be a polynomial in $r$ in the Chow cohomology. This follows from \cite{JPPZ18}*{Corollary 10} by replacing $D$ with $D\times B$. The following compatibility of the virtual cycles
\[(\tau_1)_*\left([\bM_{g,\vec{a}}(\sqrt[r]{D\times B/L},\beta))]^{\vir_\pi}\right)=r^{2g-1}[\bM_{g,n+\rho_0+\rho_\infty}(D\times B,\beta)]^{\vir_\pi}\]
still holds because
\[
[\bM_{g,\vec{a}}(\sqrt[r]{D\times B/L},\beta))]^{\vir_\pi} = [\bM_{g,\vec{a}}(\sqrt[r]{D/L_D},\beta_D))]^{\vir}\times [B]
\]
and 
\[[\bM_{g,n+\rho_0+\rho_\infty}(D\times B,\beta)]^{\vir_\pi}=[\bM_{g,n+\rho_0+\rho_\infty}(D,\beta_D)]^{\vir}\times [B]\]
where $\beta_D$ is the pushforward of $\beta$ to the first factor of $D\times B$.



\begin{rmk}
In the argument of \cite{TY}, there are extensive discussions about insertions. But by establishing relations between virtual cycles, the discussions about insertions can be reduced. This simplifies the proof of \cite{TY}.
\end{rmk}

\subsection{Equivariant theory as a limit of non-equivariant theory}
In this subsection, we briefly explain how Lemma \ref{lem:appfamily} implies Lemma \ref{lem:appeq}.

According to \cite{EG}*{Section 2.2}, the $i$th equivariant Chow group of a space $X$ under an algebraic group $G$ can be defined as follows. Let $V$ be a $l$-dimensional representation of $G$ with $U\subset V$ an equivariant open set where $G$ acts freely and whose complement has codimension more than $\dim(X)-i$. Then define
\begin{equation}\label{eqn:eqchow}
A^G_i(X)=A_{i+l-g}((X\times U)/G),
\end{equation}
where $\dim(G)=g$. 
In our case, $G=\GM$. We simply choose $E=U=\C^N-\{0\}$. Now $(X\times E)/\GM$ is an $X$-fibration over $B=U/G\cong \Pp^{N-1}$.

Recall that we have
\begin{align*}\label{eqn:rel}
\bM^{\rel}((P\times E)/\GM,(D_0\times E)/\GM) &\cong (\bM^{\rel}(P,D_0) \times E)/\GM.\\
\bM^{\text{orb}}((P_{D_0,r}\times E)/\GM) &\cong (\bM^{\text{orb}}(P_{D_0,r}) \times E)/\GM.
\end{align*}
For suitable $N$, \eqref{eqn:eqchow} identifies the equivariant Chow group with a non-equivariant model. Equivariant virtual class 
\[[\bM^{\rel}((P\times E)/\GM,(D_0\times E)/\GM)]^{\vir, \eq}\] is defined using equivariant perfect obstruction theory and doing intersection theory equivariantly. But equivariant procedure eventually reduces to a finite model like this. By comparing definitions side-by-side, one can check that $[\bM^{\rel}((P\times E)/\GM,(D_0\times E)/\GM)]^{\vir_\pi}$ is identified with $[\bM^{\rel}((P\times E)/\GM,(D_0\times E)/\GM)]^{\vir, \eq}$ under \eqref{eqn:eqchow}. Thus, Lemma \ref{lem:appfamily} implies Lemma \ref{lem:appeq}.

\subsection{Identifying localization residues}
Lemma \ref{lem:appeq} implies Theorem \ref{thm:appdx} by comparing localization residues.

Consider the localization residue of $\bM^{\rel}(P,D_0)$ corresponding to a vertex of class $\beta$ over $D_0$ with $\rho_\infty$ edges going with degrees $\nu_1,\ldots,\nu_{\rho_\infty}$. We also put $n$ interior markings and $\rho_0$ relative markings with contact order $(\mu_1,\ldots,\mu_{\rho_0})$. The residue is
\[
\dfrac{\text{Edge}}{t-\Psi_\infty} \cap [\bM_{\Gamma}^\sim(D)]^{\vir}.
\]
where $\text{Edge}$ stands for the edge contribution, and $t$ is the equivariant parameter. Similarly, for the localization residue of $\bM^{\text{orb}}(P_{D_0,r})$, consider the graph with one vertex over $\cD_0$ with corresponding markings and edges. Its localization residue is the following:
\[
\text{Edge}\left(\sum_{i\geq 0} c_j(-R^*\pi_*\mathcal L_r)(t/r)^{g-1-j}\right)\prod\limits_{i=1}^{\rho_\infty}\dfrac{1}{1+\dfrac{\ev_i^*c_1(L)-\nu_i\bar\psi_i}{t}}  \cap [\bM_{g,\vec{a}}(\cD_0,\beta)]^{\vir}.
\]
Note that the edge contributions in the orbifold case are the same as the ones in the relative case. Now push both localization residues forward to the corresponding fixed component of $\bM_{g,n+\rho}(P,\beta)$ ($\rho=\rho_0$). Edge contributions are pullback classes from $\bM_{g,n+\rho_0+\rho_\infty}(D,\beta)$ and can thus be cancelled. The result is precisely Theorem \ref{thm:appdx}.

\section{Cycle relations for the moduli space of stable maps}\label{sec:cyc-rel}
In this section, we derive some relations of cycle classes in the moduli of stable maps using Theorem \ref{thm:appdx}.
The cycle relations that we are obtaining are related with double ramification cycles with target varieties. A nice introduction to double ramification cycles can be found in \cite{JPPZ}*{Section 0}.

Double ramification cycles (without target varieties) are computed in \cite{JPPZ}.
Double ramification cycles with target varities have been studied in \cite{JPPZ18}. Let $D$ be a smooth projective variety. Let $L\rightarrow D$ be a line bundle. Then we have a $\mathbb P^1$-bundle over $D$:
\[
\mathbb P(\mathcal O_D\oplus L)\rightarrow D.
\]
Recall that $\bM_{\Gamma}^\sim(D)$ is the moduli of relative stable maps to rubber
targets over $D$. 
The double ramification data is a vector $\vec a=(a_1,\ldots,a_{\rho_0+\rho_\infty+n})\in \mathbb Z^{\rho_0+\rho_\infty+n}$ of $\rho_0+\rho_\infty+n$ integers such that
\[
\sum_{i=1}^{\rho_0+\rho_\infty+n} a_i=\int_\beta c_1(L).
\]
The double ramification cycle with the target variety $D$ is defined as the pushforward
\[
\DR_\Gamma(D,L):=\tau_*[\bM_{\Gamma}^\sim(D)]^{\operatorname{vir}}\in A_{\on{vdim}-g}(\bM_{g,\rho_0+\rho_\infty+n,\beta}(D)),
\]
where
\[
\on{vdim}=(1-g)(\dim_{\mathbb C} D-3)+\int_\beta c_1(D)+\rho_0+\rho_\infty+n.
\]

Double ramification cycles computed in \cite{JPPZ} and \cite{JPPZ18} are related with tautological classes $P_{\Gamma}^{d,r}(D,L)$ on $\bM_{g,\rho_0+\rho_\infty+n,\beta}(D)$. We briefly describe $P_{\Gamma}^{d,r}(D,L)$ in here. 

Following \cite{JPPZ18}*{Section 0.3}, we consider the $D$-valued stable graphs. 

\begin{defn}
The set $G_{g,m,\beta}(D)$ of $D$-valued stable graphs consists of graphs $\Phi\in G_{g,m,\beta}(D)$ with the data
\begin{align*}
\left\{V(\Phi), H(\Phi), L(\Phi),\g: V(\Phi)\rightarrow \mathbb Z_{\geq 0}, v: H(\Phi)\rightarrow V(\Phi),\right. \\
\left.\iota: H(\Phi)\rightarrow H(\Phi), \beta: V(\Phi)\rightarrow H_2(D,\mathbb Z)\right\}
\end{align*}
satisfying the following properties:
\begin{enumerate}
    \item $V(\Phi)$ is the set of vertices. The set $V(\Phi)$ is equipped with a genus function $\g: V(\Phi)\rightarrow \mathbb Z_{\geq 0}$.
    \item $E(\Phi)$ is the set of edges. It is defined by the $2$-cycles of $\iota$ in $H(\Phi)$. Note that we allow self-edges at vertices.
    \item $L(\Phi)$ is the set of legs. It is defined by the fixed points of $\iota$. The set $L(\Phi)$ corresponds to the set of $m$ markings.
    \item The genus condition: the pair $(V(\Phi),E(\Phi))$ defines a connected graph satisfying
    \[
    \sum_{v\in V(\Phi)}\g(v)+h^1(\Phi)=g.
    \]
    \item The stability condition: for each vertex $v$ with $\beta(v)=0$, we have
    \[
    2\g(v)-2+m(v)> 0,
    \]
    where $m(v)$ is the valence of $\Phi$ at $v$ including both edges and legs.
    \item The degree condition: 
    \[
    \sum_{v\in V(\Phi)}\beta(v)=\beta.
    \]
\end{enumerate}
\end{defn}
A $D$-valued stable graph $\Phi$ determines a moduli space $\bM_\Phi$ of stable maps. There is a canonical map
\[
j_\Phi: \bM_\Phi\rightarrow \bM_{g,m,\beta}(D).
\]

There are some tautological classes on $\bM_{g,k,\beta}(D)$ obtained from the universal curve
\[
\pi: \mathcal C_{g,m,\beta}(D)\rightarrow \bM_{g,m,\beta}(D). 
\]
Let $s_i$ be the $i$th section of the universal curve and 
\[
f:\mathcal C_{g,m,\beta}(X)\rightarrow X
\]
be the universal map.
Let $\omega_\pi$ be the relative dualizing sheaf and $\omega_{\on{log}}$ be the relative logarithmic line bundle. We have the following classes
\[
\xi_i=c_1(s_i^*f^*L), \quad \eta_{a,b}=\pi_*(c_1(\omega_{\on{log}})^a\xi^b),
\]
where $\xi=c_1(f^*L)$.
\begin{defn}
A weighting mod $r$ of $\Phi$ is a function on the set of half-edges,
\[
w: H(\Phi)\rightarrow \{0,1,\ldots, r-1\},
\]
satisfying the following properties 
\begin{enumerate}
    \item for each $i\in L(\Phi)$, corresponding to the marking $i\in\{1,\ldots, m\}$,
    \[
    w(i)=a_i \mod r,
    \]
    \item for each $e\in E(\Phi)$, corresponding to two half-edges $h, h^\prime\in H(\Phi)$,
    \[
    w(h)+w(h^\prime)=0 \mod r,
    \]
    \item for each $v \in V(\Phi)$, 
    \[
    \sum_{v(h)=v}w(h)=\int_{\beta(v)}c_1(L), \mod r.
    \]
\end{enumerate}
\end{defn}
We denote by $W_{\Phi,r}$ the finite set of all possible weightings of $\Phi$ mod $r$ . Now we are ready to define the class $P_{\Gamma}^{d,r}(D,L)$ following \cite{JPPZ18}*{Section 0.6}. The class $P_{\Gamma}^{d,r}(D,L)$ is the degree $d$ component of the tautological class
\begin{align*}
    \sum_{\Phi\in G_{g,m,\beta}(D),w\in W_{\Phi,r}}\frac{r^{-h^1(\Phi)}}{|\on{Aut}(\Phi)|}j_{\Phi *}\left[ \prod_{i=1}^m\exp\left(\frac 12 a_i^2\psi_i+a_i\xi_i\right)\prod_{v\in V(\Phi)}\exp \left(-\frac 12\eta_{0,2}(v)\right)\right.\\
\left. \prod_{e=(h,h^\prime)\in E(\Phi)} \frac{1-\exp\left(-\frac{w(h)w(h^\prime)}{2}(\psi_h+\psi_{h^\prime})\right)}{\psi_h+\psi_{h^\prime}}   \right]
\end{align*}

By \cite{JPPZ18}*{Proposition 1}, the class $P_{\Gamma}^{d,r}(D,L)$ is a polynomial in $r$ for all sufficiently large $r$. The constant term of the polynomial is denoted by $P_{\Gamma}^{d}(D,L)$. \cite{JPPZ18}*{Theorem 2} states that the double ramification cycle exactly equals to class $P_{\Gamma}^{g}(D,L)$. 


For the constant term of $P_{\Gamma}^{d,r}(D,L)$, we have the following proposition.
\begin{prop}[\cite{JPPZ18}]
The two cycle classes \[r^{2d-2g+1}(\tau_1)_*\left(c_d(-R^*\pi_*\mathcal L_r)\cap [\bM_{g,\vec{a}}(\sqrt[r]{D/L},\beta)]^{vir}\right)\]
and $P_{\Gamma}^{d,r}(D,L)$ have the same constant term.
\end{prop}

Therefore, Corollary \ref{cor:cycrel} implies the following vanishing theorem which is also proved in \cite{CJ} and \cite{Bae19}.

\begin{thm}[\cite{CJ}*{Theorem 1.2}, \cite{Bae19}*{Theorem 3.4}]
For $d>g$, we have
\[
P_{\Gamma}^{d}(D,L)=0.
\]
\end{thm}

The refined polynomiality in Corollary \ref{cor:cycrel} suggests that further vanishing results might also be true for $P_{\Gamma}^{d,r}(D,L)$ when $d>g$.



\bibliographystyle{amsxport}
\bibliography{universal-BIB}
\end{document}